\documentclass[a4paper,10pt]{article}
\usepackage[utf8]{inputenc}
\usepackage{amsmath,amsfonts,amssymb}
\usepackage{enumerate}
\usepackage{graphics}
\usepackage{graphicx}
\usepackage{epsfig}
\usepackage{cite}
\usepackage{url}

\newtheorem{algorithm}{Algorithm}
\newtheorem{example}{Example}
\newtheorem{theorem}{Theorem}
\newtheorem{lemma}{Lemma}

\newtheorem{definition}{Definition}
\newtheorem{proposition}{Proposition}
\newtheorem{proof}{Proof.}

\newcommand*{\QEDA}{\hfill\ensuremath{\blacksquare}}%
\newcommand{\xbold}{{\bf x}}
\newcommand{\zbold}{{\bf z}}
\newcommand{\Qbold}{{ Q}}

\newcommand{\ybold}{{\bf y}}

\newcommand{\ebold}{{\bf e}}
\newcommand{\fbold}{{ f}}

\newcommand{\dbold}{{\bf d}}
\def\d{\mathbf d}
\def\x{\mathbf x}
\def\y{\mathbf y}

\def\e{\mathbf e}

\newcommand{\R}{{\mathbb R}}
\renewcommand{\Re}{{\mathbb R}}
\def\trt{^{\scriptscriptstyle T}}

\title{A Class of  Parallel Decomposition Algorithms for SVMs Training\thanks{The work of Laura Palagi was partially supported by the italian project PLATINO (Grant Agreement n. PON01\_01007); the work of Simone Sagratella was partially supported by the grant: Avvio alla Ricerca 397, Sapienza University of Rome}}
\date{}
\author{Andrea Manno\footnote{Dipartimento di Ingegneria dell'Informazione, Università degli Studi Firenze,Via Santa Marta, 3 50139, Florence, Italy. (andrea.manno@unifi.it)}, Laura Palagi\footnote{Department of Computer, Control and Management Engineering, Sapienza University of
Rome, Via Ariosto 25, 00185, Rome Italy. (palagi@dis.uniroma1.it)} and Simone Sagratella\footnote{Department of Computer, Control and Management Engineering, Sapienza University of
Rome, Via Ariosto 25, 00185, Rome Italy. (sagratella@dis.uniroma1.it)}}

\begin{document}
\maketitle

\begin{abstract}
The training of Support Vector Machines may be
a very difficult task when dealing with very large
datasets. The memory requirement and the time consumption of the SVMs algorithms
grow rapidly with the increase of the data. To overcome
these drawbacks, we propose a parallel decomposition algorithmic
scheme for SVMs training for which we prove global convergence
under suitable conditions. We outline how these assumptions
can be satisfied in practice and we suggest various
specific implementations exploiting the adaptable
structure of the algorithmic model.

\end{abstract}
\noindent \textbf{Keywords.} Decomposition Algorithm, Big Data, Support Vector Machine, Machine Learning, Parallel Computing

\section{Introduction}
A Support Vector Machine (SVM) is a well known
classification and regression tool that has
spread in many scientific fields during the last two decades, see \cite{SVM}.
Given a training set of $n$ input-target pairs
$$ D = \{ (\zbold_r,y_r), \; r=1,\dots,n, \; \zbold_r \in \Re ^m, \; y_r \in \{-1,1\}  \}, $$
an SVM provides a prediction model used to classify new unlabeled samples.\\
The dual formulation of an SVM training problem is
\begin{align}\label{eq:problem}
\displaystyle\min_{\xbold} \; & f (\xbold) := \frac{1}{2} \xbold\trt  Q \xbold - \ebold\trt \xbold \nonumber\\
        & \ybold\trt \xbold = 0 \\
        & \mathbf{0} \leq \xbold \leq C \ebold,\nonumber
\end{align}
where $\xbold \in \Re^n$, $\ebold \in \Re^n$ is a vector of all ones,
$C>0$ is a positive constant, $\ybold \in \{-1,1\}^n$ and $Q$ is an $n \times n$ symmetric positive semidefinte matrix.
Each component of $\xbold$ is associated with a sample
of the training set and $\ybold$ is the vector of the corresponding labels.
Entries of $Q$ are defined by
$$ Q_{rq}=y_r y_q K(\zbold_r,\zbold_q), \quad\quad r,q=1,2,\dots,n, $$
where $K:\Re^m\times\Re^m\to\Re$ is a given kernel function \cite{Cristianini}.\\
Many real SVM applications are characterized
by a large dimensional training set. This implies that hessian matrix $Q$ is so big that it cannot
be entirely stored in memory. For this reason
classical optimization algorithms that use first and second
order information cannot be used to efficiently solve problem \eqref{eq:problem}.\\
To overcome this difficulty, many decomposition algorithms have been
proposed in literature. At each iteration, they split the original problem
into a sequence of smaller subproblems where only a subset of variables (working set) are updated. Columns of the
hessian submatrix corresponding to each subproblem are, partially or entirely, recomputed
at each step.
These strategies can be mainly divided into
SMO (Sequential Minimal Optimization) and non-SMO methods. SMO algorithms (see e.g. \cite{LIBSVM,PlattSMO}) work
with subproblems of dimension two, so that their solutions
can be computed analytically; while non-SMO algorithms (see e.g. \cite{LIBLINEAR,SVMlight})
need an iterative optimization method to solve each subproblem.
From the theoretical point of view, the policy for updating the working set plays a crucial role to prove convergence.
 In case of SMO methods, a proper selection rule
based on the maximal violating principle is sufficient to ensure asymptotic
convergence of the decomposition scheme \cite{lin:ieee06,lin:ieee02}. For larger working sets,
convergence proofs are available under further conditions  \cite{SVMlightconvergence,hybrid,Palagi2005311}.

In recent years SVMs have been applied to huge datasets, mainly related
to web-oriented applications. To reduce the big
amount of time needed for the training of SVMs on such huge datasets, parallel algorithms have been proposed.
Some of these  parallel approaches to SVMs
consists in distributing
the most expensive tasks, such as subproblems solving and gradient updating,
 among the available processors, see \cite{Ferris2003783,Zanni,Zanni2}.
Another way of fruitfully exploit parallelism is based on splitting the training data
into subsets and distributing them among the processors \cite{parSMO1,ParSMO}.
Among these parallel techniques, there are also the so called Cascade-SVM (see \cite{Cascade,impcascade}) that has been introduced to face
big dimensional instances. While achieving a good reduction
of the training time respect to sequential methods, these methods may
lack convergence properties or may require strong assumptions to prove it.

Actually, combining decomposition rules, for the selection of working sets, and parallelism makes the proof of
convergence a very difficult task, see \cite{Liuzzipalagipiacentini}. This is mainly due to nonseparability of the feasible set of problem \eqref{eq:problem}.

In this work we propose a class of convergent parallel training algorithms
based on the decomposition of problem \eqref{eq:problem} into
a partition of subproblems that can be solved independently by parallel
processes. The convergence to a global optimum of problem \eqref{eq:problem} is proved
under realistic assumptions.
It partially exploits
results introduced in \cite{flexa,tseng}.
The algorithmic framework presented may include, as a special case, other convergent
theoretical models like \cite{Liuzzipalagipiacentini}.

The paper is organized as follows: in Section \ref{sec: preliminaries} we introduce some preliminary results;
in Section \ref{sec: alg} we introduce a general parallel algorithmic scheme. We analyze its convergence
properties in Sections \ref{sec:convergence}, \ref{sec:partition} and \ref{sec:realistic}; in Section \ref{sec:implementation}
we discuss about some possible practical implementations.

\paragraph{Notation}
In the following we use this notation. Vectors are boldface. Given a vector $\x \in \R^n$ with components $x_r$ and a subset of indices $P\subseteq \{1,\dots,n\}$ we denote by
$\x_{P} \in \R^{|P|}$ the subvector made up of components $x_r$ with $r\in P$ and by
$\x_{-P} \in \R^{n-|P|}$ the subvector made up of components $x_r$ with $r\not\in P$.
By $\|\cdot\|$ we indicate the euclidean norm, whereas the zero norm of a vector $\|\x\|_0$ denotes the number of nonzero components of vector $\x$.
Further given a square $n\times n$ matrix  $\Qbold$, we denote by
$\Qbold_{* r}$ the $r-$th column of the matrix. Given two subsets of indices
$P_r, P_q\subseteq \{1,\dots,n\}$,  we write $Q_{P_r P_q}$ to indicate the $|P_r|\times | P_q|$ submatrix of $Q$ with row indices in block $P_r$ and column indices in block $P_q$. We denote by $\lambda_{\min}^Q$ and $\lambda_{\max}^Q$ respectively the minimum and maximum eigenvalue of a square matrix $Q$.
For the sake of simplicity we denote the $r-$th component of the gradient as $\nabla f(\x)_r=\frac{\partial f (\x)}{\partial x_r}$ and as  $\nabla_P f(\x)\in \R^{|P|}$ the subvector of the gradient made up of components $\frac{\partial f (\x)}{\partial x_r}$ with $r\in P$.
We denote by $\cal F$ the feasible set of problem \eqref{eq:problem}, namely
$${\cal F}=\{\x\in \R^n: \ybold\trt \xbold = 0, \ \  \mathbf{0} \leq \xbold \leq C \ebold\}.$$
Note that all the results that we report in the sequel hold also in the case of feasible set
${\cal F}=\{\x\in \R^n: \ybold\trt \xbold = b, \ \  \mathbf{0} \leq \xbold \leq C \ebold\},$
where $\y\in \R^n$ and $b\in \R$, but for sake of simplicity we refer to the case $b=0$ and $\y\in\{-1,1\}^n$.

\section{Optimality Conditions and Preliminary Results}\label{sec: preliminaries}
Let us consider a solution $\xbold^*$ of problem \eqref{eq:problem}. Since constraints are linear and the objective function is convex,  necessary and sufficient conditions for optimality are the Karush-Kuhn-Tucker (KKT) conditions that state that
there exists a scalar $s$ such that for all indices $r \in \{1,\ldots,n\}$:
\begin{align}\label{pre-KKT}
\nabla f(\xbold^*)_r+ s y_r  \geq 0 \quad \mbox{if} &\ x^*_r=0 \nonumber\\
\nabla f(\xbold^*)_r+s y_r  \leq 0 \quad \mbox{if} &\ x^*_r=C \\
\nabla f(\xbold^*)_r+s y_r  =0 \quad \mbox{if} &\ 0< \ x^*_r< C. \nonumber
\end{align}
It is well known (see e.g.\cite{SVMlight}) that KKT conditions can be written in a more compact form by introducing the following sets
\begin{equation*}
 I_{up}(\xbold) := \{r \subseteq \{1,\dots,n\} : \ x_r<C, \ y_r=1,\mbox{ or } x_r>0, \ y_r=-1 \},
\end{equation*}
\begin{equation*}
 I_{low}(\xbold) := \{r \subseteq \{1,\dots,n\} : \ x_r<C, \ y_r=-1,\mbox{ or } x_r>0, \ y_r=1 \}.
\end{equation*}
Assuming that $I_{up}(\xbold^*) \neq \emptyset$ and $I_{low}(\xbold^*) \neq \emptyset$, then we can rewrite \eqref{pre-KKT} as
\begin{equation}\label{optimality}
m(\xbold^*)=\max_{r \in I_{up}(\xbold^*)} -\frac{\nabla f(\xbold^*)_r}{y_r} \leq  \min_{r \in I_{low}(\xbold^*)} -\frac{\nabla f(\xbold^*)_r}{y_r}=M(\xbold^*).
\end{equation}
By the convexity of problem \eqref{eq:problem}, we can say that $\xbold^*$ is optimal if and only if
either $I_{up}(\xbold^*) = \emptyset$ or $I_{low}(\xbold^*) = \emptyset$ or condition \eqref{optimality} holds.

Such a form of the KKT conditions is the basis of most efficient sequential decomposition algorithms for the solution of problem \eqref{eq:problem}. In decomposition algorithms the sequence $\{\xbold^k\}$  is obtained by changing at each iteration only a subset of the variables, let's say $\xbold_{P_i}$ with $P_i\subset\{1,\ldots,n\}$, whilst the other $\xbold_{-P_i}$ remain unchanged. Thus the sequence takes the form
$$\x^{k+1}=\x^k+\alpha^k \d^k,$$ where $\d^k$ is a sparse feasible descent direction such that $\|\d^k\|_0=|P_i|$ with $|P_i|<<n$ and $\alpha^k$ represents a stepsize along this direction. Whatever the feasible direction $\d^k$ is, since the objective function is quadratic and convex, the choice
of the stepsize can be performed by using an exact minimization of the objective function along  $\d^k$.
Indeed, let $\bar\beta>0$ be the largest feasible step at $\x^k\in {\cal F}$ along the descent direction $\d^k$ then
\begin{equation}\label{eq:alfa_ott}
  \alpha^k :=
  \min \left\{
  - \frac{\nabla \fbold(\x^k)\trt \d^k}{{\d^k}\trt \Qbold  \d^k}, \, \bar\beta
  \right\}.
\end{equation}
Sequential decomposition methods differ in the choice of the direction $\d^k$, or equivalently in the choice of the so called working set $P_i$.

Sequential Minimal Optimization (SMO) methods uses feasible descent directions $\d^k$ with $\|\d^k\|_0=2$ which is the minimal possible cardinality due to the equality constraint.
In a feasible point ${\xbold}\in {\cal F}$ a feasible direction with two nonzero components $\d^{(ij)}$ is given by
\begin{align}\label{smo_generic}
 d^{(ij)}_r := \begin{cases} \displaystyle\ \ \frac{1}{y_r} & \mbox{ if } r=i  \\[.8em]
\displaystyle -\frac{1}{y_r} & \mbox{ if } r=j   \\[.8em] \quad 0 & \mbox{ otherwise}\end{cases}, \quad r = 1, \ldots, n.
\end{align}
for any pair $(i,j)\in I_{up}(\x) \times  I_{low}(\x)$. We say that a pair $(i,j)\in I_{up}(\x) \times  I_{low}(\x)$ is a violating pair at $\x$ if it satisfied also $\nabla f(\x)^T\d^{(ij)}<0$.

The exact optimal stepsize $\alpha \ge 0$ along a direction $\d^{(ij)} $ can be efficiently computed by noting
that in \eqref{eq:alfa_ott} we have \begin{equation}\label{eq:t}
\bar \beta= \min \left\{\beta_{i}, \, \beta_{j}\right\},\end{equation}
where
\begin{equation}\label{eq: alfa_feas}
\beta_{h} := \left\{ \begin{array}{ll}
 x_{h} & \text{ if } d^{(ij)}_{h} < 0 \\
 C - x_{h} & \text{ if } d^{(ij)}_{h} > 0.\end{array}\right.\end{equation}
Thus we get the value of the optimal stepsize $\alpha$ along a direction $\d^{(ij)}$ as
\begin{equation}\label{eq:alfa_smo}
  \alpha :=
  \min \left\{
  - \frac{\nabla \fbold_{i}y_{i}-\nabla \fbold_{j}y_{j}}{Q_{ii}+Q_{jj}-2y_{i}y_{j}Q_{ij}}
, \, \bar \beta
  \right\}.
\end{equation}
Among such minimal descent directions, i.e. violating pairs,  a crucial role is played by the so called {\it Most Violating Pair} (MVP) direction (see e.g. \cite{SVMlight}). To be more specific,
 given a feasible point
${\xbold}$, let us define the sets
$$
I_{up}^{MVP}({\xbold}):=\left\{i\in I_{up}(\x):~i\in\arg\!\!\!\max_{h\in I_{up}({\xbold})} -\frac{\nabla
f({\xbold})_h}{y_h}\right\},
$$
$$
I_{low}^{MVP}({\xbold}):=\left\{j\in I_{low}(\x):~j\in\arg\!\!\!\min_{h\in
{I_{low}}({\xbold})} -\frac{\nabla
f({\xbold})_h}{y_h}\right\}.
$$
If $\x$ is not a solution of problem \eqref{eq:problem}, then $(i_\texttt{MVP},j_\texttt{MVP})\in I_{up}^{MVP}(\xbold)\times
I_{low}^{MVP}(\xbold)$ is a pair, possibly not unique, that violates the  KKT conditions at most and it is said {\it Most Violating
Pair} (MVP). In the sequel, for the sake of notational simplicity, we assume that, for every feasible $\x$, the MVP is unique as this makes no difference in our analysis.\\
The direction $\d_{\texttt{MVP}} \in \Re^n$ corresponding to the pair $(i_\texttt{MVP},j_\texttt{MVP})\in I_{up}^{MVP}(\xbold)\times
I_{low}^{MVP}(\xbold)$ is, among all feasible descent directions with only two nonzero components, the steepest descent one at $\x$. \\
Now we are ready to introduce the  definition of ``most violating step''. Let $\x_{\texttt{MVP}} = \x + \alpha_{\texttt{MVP}} \d_{\texttt{MVP}}$ with $\alpha_{\texttt{MVP}}$ obtained by \eqref{eq:alfa_smo} with $i=i_\texttt{MVP}$, $j=j_\texttt{MVP}$.
\begin{definition}[Most Violating Step]
At a point $\x \in {\cal F}$, we define the ``Most Violating Step'' (MVS) $S_{\texttt{MVP}}$ as:
\begin{equation}\label{eq: MVP stepsize}
 S_{\texttt{MVP}}(\x) := \| \x_{\texttt{MVP}}-\x\|= |\alpha_{\texttt{MVP}}| \|\d_{\texttt{MVP}}\|.
\end{equation}
\end{definition}
In particular, since $y^i\in\{-1,1\}$ we have that $S_{\texttt{MVP}}(\x) = |\alpha_{\texttt{MVP}}| \sqrt{2}$.

We can state the optimality condition using the definition of MVS.
\begin{proposition}\label{pr: M=0}
 A point $\xbold^* \in \cal F$  is optimal for problem \eqref{eq:problem} if and only if either $I_{up}(\xbold^*) = \emptyset$ or $I_{low}(\xbold^*) = \emptyset$ or $S_{\texttt{MVP}}(\xbold^*)=0$.
\end{proposition}
\begin{proof}
 As said above $\xbold^*$ is optimal for problem \eqref{eq:problem} if and only if either $I_{up}(\xbold^*) = \emptyset$ or $I_{low}(\xbold^*) = \emptyset$ or condition \eqref{optimality} holds.
 Therefore we only have to show that, in the case in which $I_{up}(\xbold^*) \neq \emptyset$ and $I_{low}(\xbold^*) \neq \emptyset$, the following holds:
 $$
 S_{\texttt{MVP}}(\xbold^*)=0 \quad \Leftrightarrow \quad m(\xbold^*)\leq M(\xbold^*).
 $$
 Since $I_{up}(\xbold^*) \neq \emptyset$ and $I_{low}(\xbold^*) \neq \emptyset$, we can compute a pair $(i_{\texttt{MVP}}^*, j_{\texttt{MVP}}^*)\in I_{up}^{MVP}(\xbold^*)\times
I_{low}^{MVP}(\xbold^*)$ and $\d_{\texttt{MVP}}^*=
 \d^{(i_{\texttt{MVP}}^*, j_{\texttt{MVP}}^*)}$ as in \eqref{smo_generic}.
 $m(\xbold^*)\leq M(\xbold^*)$ is equivalent to inequality $\nabla \fbold(\xbold^*)\trt \d_{\texttt{MVP}}^*\ge 0$.
 By noting that $\d_{\texttt{MVP}}^*
 $ is a feasible direction at $\x^*$, then from \eqref{eq:t} we have $\bar\beta>0$.
 Therefore by \eqref{eq:alfa_smo} we can conclude that $\nabla \fbold(\xbold^*)\trt \d_{\texttt{MVP}}^* \ge 0$
 if and only if
 $\alpha_{\texttt{MVP}}^*=0$
 and, in turn, if and only if $S_{\texttt{MVP}}(\xbold^*)=0$, so that the proof is complete. \QEDA
\end{proof}

\section{A Parallel Decomposition Model}\label{sec: alg}

In this section we introduce a parallel decomposition scheme for finding a solution of problem \eqref{eq:problem}. The theoretical properties and implementation details are discussed in the next sections.
The algorithm fits in a decomposition framework where, as usual, the solution of problem \eqref{eq:problem} is obtained by a sequence of solution of smaller problems in which only a subset of the variables is changed.
 To fix notation, let $\x^k \in \cal F$ and consider a subset $P_i\subset \{1,\dots, n\}$, so that $\x^k $ can be partitioned as
$\xbold^k := (\x^k_{P_i}, \xbold^k_{{-P_{i}}})$.
The problem of minimizing  over $\x_{P_i}$ with $\xbold_{{-P_{i}}}$ fixed to the current value $\xbold_{{-P_{i}}}^k$ is:
\begin{align}\label{eq: problem i-th proximal}
\displaystyle\min_{\x_{P_i} \in {\cal F}_{P_i}^k} \; & f_{P_i}(\x_{P_i}, \xbold_{{-P_{i}}}^k) + \frac{\tau^k_i}{2} \|\x_{P_i} - \x_{P_i}^k\|^2,
\end{align} where a proximal point term with $\tau^k_i\ge 0$ has been added \cite{Palagi2005311},
and the feasible set is
$$ {\cal F}_{P_i}^k:= \{ \x_{P_i} \in \Re^{|P_i|}: \ybold_{P_i}\trt \x_{P_i}=\ybold_{{P_i}}\trt  \x_{{P_{i}}}^k,\; 0 \le \x_{P_i} \le C \ebold_{P_i} \}.
$$
Problem \eqref{eq: problem i-th proximal} is still quadratic and convex with hessian matrix $\Qbold_{P_iP_i}+\tau^k_i I_{P_{i}} $ symmetric and positive semidefinite, and linear term given by $\ell^k_{P_i}+\tau^k_ix_{P_i}^k$, where
$$\ell^k_{P_i}= \sum_{P_j \in {\cal P}, j \neq i} \Qbold_{{P_i}{P_j}}  \x_{P_j}^k - \ebold_{P_i}.$$
We denote $\widehat \x^k_{P_i}$ as a solution of problem \eqref{eq: problem i-th proximal}, which is unique either if $\tau^k_i>0$ or if $\Qbold_{P_i P_i}$ is positive definite.

The parallel scheme that we are going to define is not based on splitting the data set or in parallelizing the linear algebra, but on defining a bunch of subproblems to be solved by means of  parallel and independent processes.  
Unlike sequential decomposition methods, the
search direction $\d^k$ is obtained by summing up smaller directions obtained by solving in parallel a bunch of subproblems of type \eqref{eq: problem i-th proximal}.

Let us define a partition ${\cal P}=\{P_1,P_2,\ldots\}$ of the set of all indices $\{1, \ldots, n\}$. By definition we have that $P_i\cap P_j=\emptyset $ and $\cup_{i} P_i = \{1, \ldots, n\}$.
The basic idea underlying the definition of the parallel decomposition algorithm is summarized in the scheme below.
\begin{center}
\framebox{\newlength{\miniwidthb}\setlength{\miniwidthb}{\textwidth}
    \addtolength{\miniwidthb}{-4pc}
    \begin{minipage}{\miniwidthb}
 \begin{algorithm}\label{al: svm decomposition}{\rm
     \begin{center}
        {\bf Parallel Decomposition Model}
      \end{center}
      \begin{description}
      \item[{\bf Initialization}]  Choose $\xbold^0\in {\cal F}$ and set $k=0$.
\item[{\bf Do while}] (a stopping criterion satisfied)

  \begin{description}
    \item[\bf{S.1} (Partition definition)] \item[\hskip .5truecm] Set ${\cal P}^k=\{P_1,P_2,\ldots,P_{N^k}\}$ and set $\tau^k_i \geq 0$ for all $i=1,\ldots, N^k$.
   \item[\bf{S.2} (Blocks selection)]
   \item[\hskip .5truecm] Choose a subset of blocks ${\cal J}^k \subseteq {\cal P}^k$.

  \item[\bf{S.3} (Parallel computation)]
  \item[\hskip .5truecm] For all $P_i \in {\cal J}^k$ {compute in parallel} the optimal solution $\widehat \x^k_{P_i}$ of problem \eqref{eq: problem i-th proximal}.
 \item[\bf{S.4} (Direction)] Set $\d^k \in \Re^n$ block-wise as
 \begin{equation}\label{searchd}
  \d_{P_i}^{k}=\begin{cases}\widehat \x^k_{P_i}-\x_{P_i}^k & \text{if } P_i \in {\cal J}^k,\\
 {\mathbf 0} & \text{otherwise}.\\
\end{cases}
 \end{equation}

  \item[\bf{S.5} (Stepsize)] Choose a suitable stepsize $\alpha^k > 0$.
  \item[\bf{S.6} (Update)] Set $\x^{k+1} = \x^k + \alpha^k \dbold^k$ and $k = k + 1$.
  \end{description}

\item[{\bf End While}]
\item[{\bf Return}] $\x^k$.
\end{description}}\end{algorithm}
\end{minipage}
  }
\end{center}

\medskip\noindent
The scheme above encompasses different possible algorithms depending on the choice of the partition ${\cal P}^k$ at {\bf S.1}, the blocks selection ${\cal J}^k$ at {\bf S.2} and the stepsize rule at {\bf S.5}.
 
 A widely used standard  feasible point is 
 $\xbold^0={\mathbf 0}$, but of course different choices are possible if available. The choice of $\xbold^0={\mathbf 0}$ presents the advantage that also the gradient is available being  $\nabla f(\x^0)=-\e$.

 Checking optimality of the current point $\x^k$ may require
 zero or first order information depending on the stopping criterion adopted.
A standard stopping criterion is based on checking condition $m(\x^k)\le M(\x^k)-\eta$, for a given tolerance $\eta>0$. In this case the updated gradient $\nabla f(\x^{k+1})$ is needed at each iteration.
It is well known that for large scale problem this is a big effort due to expensive kernel evaluations. Indeed we have the following iterative updating rule
$$\nabla f(\x^{k+1})=\nabla f(\x^{k})+\alpha^k\sum_{P_i\in {\cal J}^k}\sum_{h\in P_i} Q_{*h} d_{h}^{k}.$$
At {\bf S.1} a partition ${\cal P}^k$ of $\{1,\ldots,n\}$   is defined. We point out that both the number $N^k$ of blocks and their composition in ${\cal P}^k$ can vary from one iteration to another. For notational simplicity we omit dependency of blocks $P_1,P_2,\ldots,P_{N^k}$ on the iteration $k$.  
As usual in decomposition algorithms, a correct choice of the partition is crucial for proving global convergence of the method.

At {\bf S.2} a subset ${\cal J}^k$ of blocks in ${\cal P}^k$ is selected. These blocks are the only ones used at {\bf S.3} to compute
a search direction $\dbold^k$ according to \eqref{searchd}. The selection of blocks makes the algorithmic scheme more flexible since one can set the overall computational burden.

At {\bf S.3}  we obtain an optimal solutions $\widehat \x^k_{P_i}$ of problem \eqref{eq: problem i-th proximal}
for each $P_i \in {\cal J}^k$. Note that $\widehat \x^k_{P_i}$ satisfies the optimality condition
\begin{equation}\label{eq: dirder}
\left[\nabla_{{P_i}}f_{P_i}(\widehat \x^k_{P_i},\xbold^k_{-P_{i}})+\tau^k_i (\widehat \x^k_{P_i} - \x_{P_i}^k) \right]\trt \d_{P_i}\ge 0\end{equation}
for any feasible direction $\d_{P_i}$ at $\widehat \x^k_{P_i}$.
The computational burden of this step consists in
\begin{itemize}
 \item computing vector $\ell^k_{P_i}$ to construct the objective function of \eqref{eq: problem i-th proximal} for all blocks $P_i\in{\cal J}^k$,
 \item solving the $|{\cal J}^k|$ subproblems.
\end{itemize}
These $|{\cal J}^k|$ convex quadratic problems can be distributed to different processes in order to be solved in a parallel fashion.\\
At {\bf S.4} the algorithm computes search direction $\d^k$.\\
At {\bf S.5} the algorithm computes stepsize $\alpha^k$.
We show in the next section (Theorems \ref{th: convergence1}, \ref{th: convergence2} and \ref{th: convergence3}) that, in order to have convergence of the algorithm, $\alpha^k$ can be computed according to a simple diminishing stepsize rule or a linesearch procedure (including the exact minimization rule).

\section{Theoretical Analysis}\label{sec:convergence}

In this section we analyze the theoretical properties of Algorithm \ref{al: svm decomposition}.
To this aim, we first introduce the definition of descent block and of descent iteration that are crucial for the following analysis.
\begin{definition}[Descent block]\label{dblock}
Given $\epsilon > 0$. At a feasible point $\xbold^k$, 
we say that the block of variables $P_i \subseteq\{1, \ldots, n\}$ is
 a \emph{descent block} if it satisfies
\begin{equation}\label{suffstep}
\|\widehat \x_{P_i} - \x_{P_i}^k\| \geq \epsilon S_{\texttt{MVP}}(\xbold^k),
\end{equation}
where $\widehat \x_{P_i}$ is the optimal solution of the corresponding problem \eqref{eq: problem i-th proximal} and $S_{\texttt{MVP}}(\xbold^k)=\|\x^k_{\texttt{MVP}}-\xbold^k\|$.
\end{definition}
Whenever at least one descent block $P_i$ is selected in ${\cal J}^k$ at {\bf S.2} of  Algorithm \ref{al: svm decomposition}, we say that iteration $k$ is a \emph{descent iteration}.


Under the assumption that  at least one descent block is selected for optimization at {\bf S.3}
of the parallel algorithmic model,
we will prove that by using a suitable $\alpha^k$ at {\bf S.5} the sequence $\{\xbold^k\}$ produced by the algorithm satisfies
 \begin{equation}\label{eq: diminishing lim convergence}
  \displaystyle \lim_{k \to \infty} S_{\texttt{MVP}}(\xbold^k) = 0.
 \end{equation}
 We prove later on that the assumption is easy to achieve.
 We first consider the case when the stepsize $\alpha^k$ is determined by a standard Armjio linesearch procedure along the direction $\dbold^k$.

\begin{theorem}\label{th: convergence1}
 Let $\{\xbold^k\}$ be the sequence generated by Algorithm \ref{al: svm decomposition} where $\alpha^k \le 1$ at {\bf S.5} satisfies the following Armjio condition
   \begin{equation}\label{eq: armijo}
    \fbold(\xbold^k + \alpha^k \dbold^k) \leq \fbold(\xbold^k) + \theta \alpha^k \nabla \fbold(\xbold^k)\trt \dbold^k,
   \end{equation}
     with $\theta \in (0,1)$.
     Assume that for all iterations $k$
 \begin{enumerate}[\rm (i)]
  \item  a descent iteration $\tilde k$ with $ k\le \tilde k\le k+L$ for a finite $ L \ge 0$ is generated;
  \item  either $\tau^k_i \geq \underline{\tau} > 0$ or $Q_{P_i P_i} \succ 0$, for all $P_i \in {\cal J}^k$.
 \end{enumerate}
 Then either Algorithm \ref{al: svm decomposition} terminates in a finite number of iterations to a solution of problem \eqref{eq:problem}  or $\{\xbold^k\}$ admits a limit point and it satisfies 
\eqref{eq: diminishing lim convergence}.
\end{theorem}

\begin{proof}
 First of all we note that $\{\xbold^k\}$ is a feasible sequence, in fact it is sufficient to show that for all $k$ if $\xbold^k$ is feasible then also $\xbold^{k+1}$ is feasible.
 Since for all $P_i \in {\cal J}^k$ it holds that ${\y_{P_i}}\trt \widehat \x_{P_i} = {\y_{P_i}}\trt \x_{P_i}^k$, then we can write
 \begin{align*}
  \displaystyle \ybold\trt \xbold^{k+1} &= \sum_{P_i \in {\cal J}^k} {\y_{P_i}}\trt \left(\x_{P_i}^{k} + \alpha^k (\widehat \x_{P_i} - \x_{P_i}^{k}) \right) + \sum_{P_i \notin {\cal J}^k} {\y_{P_i}}\trt \x_{P_i}^{k} \\
  &= \sum_{P_i \in {\cal J}^k} {\y_{P_i}}\trt \x_{P_i}^{k} + \sum_{P_i \notin {\cal J}^k} {\y_{P_i}}\trt \x_{P_i}^{k} = \ybold\trt \xbold^{k}=0.
 \end{align*}
 Finally by noting that for all $P_i \in {\cal J}^k$ it holds that $0 \leq \widehat \x_{P_i} \leq C$, then by the convexity of the box constraints
 and since $\alpha^k\le 1$, we obtain $0 \leq \xbold^{k+1} \leq C$ and this proves the feasibility of the sequence.

 Note that $\forall P_i \in {\cal J}^k$ the direction $-\d_{P_i}^k$ is a feasible direction at $\widehat x^k_{P_i}$ hence  by \eqref{eq: dirder}  we can write
  \begin{equation}\label{eq: proof-th1-dk}
   -\left[\nabla_{{P_i}}f_{P_i}(\widehat \x^k_{P_i},\xbold^k_{-P_{i}})+\tau^k_i (\widehat \x^k_{P_i} - \x_{P_i}^k) \right]\trt \d^k_{P_i}\ge 0.
  \end{equation}
Hence it holds that
 \begin{align*}
 -\nabla_{{P_i}}f_{P_i}(\x_{P_i}^k,\xbold^k_{-P_{i}})\trt \d_{P_i}^k=
   \left( Q_{P_iP_i} \x_{P_i}^k + \ell^k_{P_i} \right)\trt (\x_{P_i}^k-\widehat \xbold^k_{P_i}) =& \nonumber\\
   \left( Q_{P_iP_i} \x_{P_i}^k + \ell^k_{P_i} \right)\trt (\x_{P_i}^k-\widehat \xbold^k_{P_i}) -
   (\x_{P_i}^k-\widehat \xbold^k_{P_i})\trt Q_{P_iP_i} (\x_{P_i}^k-\widehat \xbold^k_{P_i}) +& \nonumber\\
   (\x_{P_i}^k-\widehat \xbold^k_{P_i})\trt Q_{P_iP_i} (\x_{P_i}^k-\widehat \xbold^k_{P_i}) =& \nonumber\\
  \left( Q_{P_iP_i} \widehat \xbold^k_{P_i} + \ell^k_{P_i} \right)\trt (\x_{P_i}^k-\widehat \xbold^k_{P_i}) +& \nonumber\\
  (\x_{P_i}^k-\widehat \xbold^k_{P_i})\trt Q_{P_iP_i} (\x_{P_i}^k-\widehat \xbold^k_{P_i}) =& \nonumber\\
  -\nabla_{{P_i}}f_{P_i}(\widehat \xbold^k_{P_i},\xbold^k_{-P_{i}})\trt \d_{P_i}^k
  +& \nonumber\\
     (\x_{P_i}^k-\widehat \xbold^k_{P_i})\trt Q_{P_iP_i} (\x_{P_i}^k-\widehat \xbold^k_{P_i}) \overset{\eqref{eq: proof-th1-dk}}{\geq}& \nonumber\\
   \tau^k_i \|\widehat \xbold^k_{P_i} - \x_{P_i}^k\|^2 +& \nonumber\\
   (\x_{P_i}^k-\widehat \xbold^k_{P_i})\trt Q_{P_iP_i} (\x_{P_i}^k-\widehat \xbold^k_{P_i}) \geq& \nonumber\\
   (\tau^k_i + \lambda^{Q_{P_i P_i}}_{\min}) \|\d_{P_i}^k\|^2&,
 \end{align*}
and then we have
\begin{equation}\label{eq: suff-desc}
 \nabla_{P_i}f_{P_i}(\x_{P_i}^k,\x_{-P_i}^k)\trt \d_ {P_i}^k\le - (\tau^k_i + \lambda^{Q_{P_i P_i}}_{\min}) \|\d_{P_i}^k\|^2.
\end{equation}
We denote by $\rho^k=\min_{P_i\in {\cal J}^k} (\tau^k_i + \lambda^{Q_{P_i P_i}}_{\min}) >0$. By assumption (ii) there exists $\rho>0$ such that $\rho^k \ge \rho$ for all $k$.

 From conditions \eqref{eq: armijo} and \eqref{eq: suff-desc} we can write
  \begin{equation}\label{eq:sufficient decrease}
\fbold(\xbold^{k+1})-\fbold(\xbold^{k}) \leq
  - \alpha^k \theta \rho \|\dbold^k\|^2,
 \end{equation}
therefore sequence $\{\fbold(\xbold^{k})\}$ is decreasing.
It is also bounded below so that it converges and
\begin{equation}\label{eq:fk conveges}
\lim_{k\to \infty} (\fbold(\xbold^{k+1})-\fbold(\xbold^{k}))=0.
\end{equation}
Let $\bar \x$ be a limit point of $\{ \x^k \}$, at least one of such points exists being ${\cal F}$ compact.
 Since, by the compactness of ${\cal F}$ and by the continuity of $f$, $-\infty < \fbold(\bar \xbold) - \fbold(\xbold^0)$ for all $\xbold^0 \in {\cal F}$, then, by \eqref{eq:sufficient decrease} and \eqref{eq:fk conveges}, we can write
 \begin{equation}\label{eq: diminishing liminf convergence 1}
  \displaystyle \sum_{k=0}^\infty \alpha^k \|\dbold^k\|^2 < + \infty.
 \end{equation}
 By condition (i) we can define an infinite subsequence $\{k\}_{\tilde K}$ made up of only descent iterations.
 Then, by \eqref{eq: diminishing liminf convergence 1},
  it follows that
 \begin{equation}\label{eq: diminishing liminf convergence 1 tilde}
  \displaystyle \sum_{k=0, k\in \tilde K}^\infty \alpha^k \|\dbold^k\|^2 <+\infty.
 \end{equation}
 A standard Armijo linesearch satisfying \eqref{eq: armijo} produces at each iteration an $\alpha^k>0$ in a finite number of steps (see \cite{bertbook:95}) and this ensures that
 \begin{equation}\label{diverge tilde}
  \sum_{k=0,k\in \tilde K}^{\infty} \alpha^k = +\infty.
 \end{equation}
 By \eqref{eq: diminishing liminf convergence 1 tilde} and \eqref{diverge tilde}, we obtain
 \begin{equation*}\label{eq: diminishing liminf convergence 2}
  \displaystyle \liminf_{k \to \infty, k \in \tilde K} \|\dbold^k\| = 0.
 \end{equation*}
 Now since each ${\cal J}^k$ with $k \in \tilde K$ contains a descent block at $\xbold^k$, by \eqref{suffstep} we can conclude that
 \begin{equation*}
  \displaystyle \liminf_{k \to \infty, k \in \tilde K} S_{\texttt{MVP}}(\xbold^k) = 0,
 \end{equation*}
 and then, since $S_{\texttt{MVP}}(\xbold^k)\ge 0$ for all $k$, we can write
 \begin{equation}\label{eq: diminishing liminf convergence 3}
  \displaystyle \liminf_{k \to \infty} S_{\texttt{MVP}}(\xbold^k) = 0.
 \end{equation}
 Suppose by contradiction that $\limsup_{k \to \infty} S_{\texttt{MVP}}(\xbold^k) > 0$, then for any $\gamma>0$ sufficiently
 small we would have $S_{\texttt{MVP}}(\xbold^k) > \gamma$ for infinitely many $k$ and $S_{\texttt{MVP}}(\xbold^k) < \frac{\gamma}{2}$ for infinitely many $k$.
 Therefore, one can always find an infinite set of indices, say ${\cal N}$, having the following property:
 for any $n \in {\cal N}$, there exists an integer $i_n > n$ such that $S_{\texttt{MVP}}(\xbold^n) < \frac{\gamma}{2}$ and $S_{\texttt{MVP}}(\xbold^{i_n}) > \gamma$.
 Then it is easy to see that $\xbold^n \neq \xbold^{i_n}$ and then $\sum_{k=n}^{i_n-1} \alpha^k \|\dbold^k\| > 0$ for all
 $n \in {\cal N}$. And then $$\liminf_{n \in {\cal N}, n \to \infty} \sum_{k=n}^{i_n-1} \alpha^k \|\dbold^k\|> 0,$$
 which is in contradiction with \eqref{eq: diminishing liminf convergence 1}. Then we finally obtain \eqref{eq: diminishing lim convergence}.
 \QEDA
\end{proof}

Note that since $\fbold$ is quadratic, condition \eqref{eq: armijo} can be guaranteed by using the exact minimization rule along direction $\dbold^k$:
\begin{equation}\label{eq: exact}
 \alpha^k := \max \left\{ \min \left\{ - \frac{\nabla \fbold(\xbold^k)\trt \dbold^k}{{\dbold^k}\trt \Qbold \dbold^k}, \, \bar \alpha^k \right\}, 0 \right\},
\end{equation}
where
$$
\bar \alpha^k := \min_{i\in\{1,\ldots,n\}: \dbold^k_i \neq 0} \left\{ \bar \alpha^k_i = \left\{ \begin{array}{ll}
 (\xbold^k)_i & \text{ if } \dbold^k_i < 0 \\
 C - (\xbold^k)_i & \text{ if } \dbold^k_i > 0
 \end{array}\right. \right\}.
$$
Note that in this case it is not necessary to impose $\alpha^k\le 1$ since the feasibility is guaranteed by construction.

If $n$ is huge it may be cheaper to compute $\alpha^k$ in an alternative way in order to save costly function evaluations. In particular we propose a diminishing stepsize strategy for which we give two different convergence results based on slightly
different hypotheses.

\begin{theorem}\label{th: convergence2}
Let $\{\xbold^k\}$ be the sequence generated by Algorithm \ref{al: svm decomposition} where $\alpha^k\in(0,1]$ at {\bf S.5} satisfies the following condition
\begin{equation}\label{eq: diminishing_rule}
   \alpha^k \to 0  \mbox{ and }  \displaystyle\sum_{k=0}^{\infty} \alpha^k = +\infty.
 \end{equation}
Assume that for all iterations $k$
 \begin{enumerate}[\rm (i)]
  \item $k$ is a descent iteration;
  \item  either $\tau^k_i \geq \underline{\tau} > 0$ or $Q_{P_i P_i} \succ 0$, for all $P_i \in {\cal J}^k$;
 \end{enumerate}
Then either Algorithm \ref{al: svm decomposition} converges in a finite number of iterations to a solution of problem \eqref{eq:problem}  or $\{\xbold^k\}$ admits a limit point and \eqref{eq: diminishing lim convergence} holds.

\end{theorem}

\begin{proof}
 Note that feasibility of the sequence $\{\xbold^k\}$ and inequality \eqref{eq: suff-desc} hold, see proof of Theorem \ref{th: convergence1}.

 For any given $k \ge 0$ we can write (Descent Lemma \cite{bertbook:95}):
 \begin{align}
  & \fbold(\xbold^{k+1})-\fbold(\xbold^{k}) \leq \alpha^k \nabla \fbold(\xbold^{k})\trt \dbold^k +
  \frac{(\alpha^k)^2 \lambda_{\max}^{\Qbold}}{2} \|\dbold^k\|^2. \label{eq: descent lemma}
 \end{align}
  By using (ii) and \eqref{eq: suff-desc} we can rewrite inequality \eqref{eq: descent lemma}:
 \begin{equation}\label{eq: diminishing 1}
 \fbold(\xbold^{k+1})-\fbold(\xbold^{k}) \leq
  \alpha^k \left( - \rho + \frac{\alpha^k \lambda_{\max}^{\Qbold}}{2} \right) \|\dbold^k\|^2,
 \end{equation}
 where $\rho>0$ is the minimum among $\underline{\tau}$ and all the minimum eigenvalues of all the positive definite principal submatrices of $\Qbold$.
 Since, by \eqref{eq: diminishing_rule}, $\alpha^k \to 0$ it follows that there exist $\bar \rho > 0$ and $\bar k$ sufficiently large such that for all $k \geq \bar k$ inequality \eqref{eq: diminishing 1} implies:
 \begin{equation*}
 \fbold(\xbold^{k+1})-\fbold(\xbold^{k}) \leq
  - \alpha^k \bar \rho \|\dbold^k\|^2.
 \end{equation*}
 Since, as said in the proof of Theorem \ref{th: convergence1}, $-\infty < \fbold(\bar \xbold) - \fbold(\xbold^{\bar k})$ for all $\xbold^{\bar k} \in {\cal F}$ and any limit point $\bar \x$ of $\{\x^k\}$, in a similar way we can write
 \begin{equation}\label{alfadinf}
  \displaystyle \sum_{k=\bar k}^\infty \alpha^k \|\dbold^k\|^2 < + \infty.
 \end{equation}
 By \eqref{eq: diminishing_rule}, $\sum_{k={\bar k}}^{\infty} \alpha^k = +\infty$, and then we obtain
 \begin{equation*}
  \displaystyle \liminf_{k \to \infty} \|\dbold^k\| = 0.
 \end{equation*}
 Now since each ${\cal J}^k$ contains a descent block at $\xbold^k$, by \eqref{suffstep} we can conclude that
 \begin{equation*}
  \displaystyle \liminf_{k \to \infty} S_{\texttt{MVP}}(\xbold^k) = 0,
 \end{equation*}
 and the thesis follows under the same reasoning of the proof of Theorem \ref{th: convergence1}.\QEDA
\end{proof}

As stated in Theorem \ref{th: convergence2}, a diminishing stepsize rule requires all iterations to be descent.
In certain applications (e.g. when variables are randomly partitioned), it could be useful to relax this condition, requiring that only a subsequence of the iterations are descent, as well
as for Theorem \ref{th: convergence1}. This is formalized in the next theorem where we assume the additional mild hypothesis of monotonicity of sequence $\{\alpha^k\}$.
\begin{theorem}\label{th: convergence3}
 Let $\{\xbold^k\}$ be the sequence generated by Algorithm \ref{al: svm decomposition} where $\alpha^k\in(0,1]$ at {\bf S.5} satisfies \eqref{eq: diminishing_rule}.
 Assume that for all iterations $k$
 \begin{enumerate}[\rm (i)]
 \item a descent iteration $\tilde k$ with $ k\le \tilde k\le k+L$ for a finite $ L \ge 0$ is generated;
  \item  either $\tau^k_i \geq \underline{\tau} > 0$ or $Q_{P_i P_i} \succ 0$, for all $P_i \in {\cal J}^k$;
  \item  $\alpha^{k} \ge \alpha^{k+1}$;
 \end{enumerate}
 then either Algorithm \ref{al: svm decomposition} converges in a finite number of iterations to a solution of problem \eqref{eq:problem} or $\{\xbold^k\}$ admits a limit point and \eqref{eq: diminishing lim convergence} holds.
\end{theorem}
\begin{proof}
Following the same reasoning of Theorem \ref{th: convergence2}, inequality \eqref{alfadinf} holds.
By condition (i) we can define an infinite subsequence $\{k\}_{\tilde K}$ containing only descent iterations.
 Then by \eqref{alfadinf} it follows that
 \begin{equation}\label{eq: diminishing liminf convergence 1 tilde 3}
  \displaystyle \sum_{k={\bar k},k\in \tilde K}^\infty \alpha^k \|\dbold^k\|^2 < + \infty.
 \end{equation}
 We can write the following chain of inequalities
 \begin{align*}
  +\infty =& \sum_{k={\bar k}+1}^{\infty} \sum_{h=0}^{L-1} \alpha^{L\cdot k+h} \le L \sum_{k={\bar k}+1}^{\infty} \alpha^{L\cdot k}\le L \sum_{k={\bar k},k\in \tilde K}^{\infty} \alpha^k,
 \end{align*}
where the equality is due to \eqref{eq: diminishing_rule}, the first inequality holds by (iii) and the second inequality holds by (i) and (iii).
Then the thesis follows from the same reasoning of the proof of Theorem \ref{th: convergence1}.\QEDA
\end{proof}

\section{Construction of the Partitions}\label{sec:partition}
To make the results stated in the previous section of practical interest, the major difficulty is to ensure that an iteration is descent in the sense that at least one descent block, according to Definition \ref{dblock}, is selected.
Next lemma gives a relation between the steplenght produced optimizing over a generic block $P_i$ and the one produced optimizing over any violating pair $(\bar i,\bar j)$ belonging to $P_i$. This result will be useful in order to practically build a descent block and it is used in Theorem \ref{th: wss}.
In this section we use the simplified assumption that any principal submatrix of $Q$ of
 order $2$ is positive definite.
\begin{lemma}\label{lemma: wss}
Assume that any principal submatrix of $Q$ of
 order $2$ is positive definite.
 Let $\xbold^k$ be a feasible point for problem \eqref{eq:problem} and let $(\bar i,\bar j)\in I_{up}(\xbold^k)\times I_{low}(\xbold^k)$.
 Suppose that a block $P_i \subseteq \{1, \ldots, n\}$ exists such that
 $(\bar i,\bar j) \subseteq P_i$.
  Let $\overline \xbold^k$ be the unique solution of

  \begin{equation}\label{eq: prob SMO}
\min_{\x_{(\bar i,\bar j)} \in {\cal F}_{(\bar i,\bar j)}} \; f_{(\bar i,\bar j)} \left(\x_{(\bar i,\bar j)}, \xbold^k_{{\{1,\ldots,n\} \setminus (\bar i,\bar j)}} \right).
\end{equation}
 Then there exists a scalar $\bar \epsilon>0$ such that
 \begin{equation}\label{eq: pair descent}
  \|\widehat \xbold^k_{P_i} - \x_{P_i}^k\| \geq \bar\epsilon \|\overline \xbold^k - \xbold^k\|,
 \end{equation}
 where  $\widehat \xbold^k_{P_i}$ is a solution of problem \eqref{eq: problem i-th proximal}.
\end{lemma}

\begin{proof}
First we note that  $$\overline \xbold^k= \xbold^k+\bar \alpha \d^{(\bar i,\bar j)}$$ with $\d^{(\bar i,\bar j)}$ defined in \eqref{smo_generic} and  $\bar \alpha$ computed as in \eqref{eq:alfa_smo}.
For the sake of notation let us set $\overline \d= \d^{(\bar i,\bar j)}$.

 Since $(\bar i,\bar j) \subseteq P_i$, two cases are possible: (a) $\widehat \xbold^k_{P_i} + \mu \overline \d_{P_i} \notin {\cal F}_{P_i}$
 for all $\mu > 0$, or (b) $\mu>0$ exists such that $\widehat \xbold^k_{P_i} + \mu \overline \d_{P_i} \in {\cal F}_{P_i}$, that is $\overline \d_{P_i}$ is a feasible direction at $\widehat \xbold^k_{P_i}$.
 \begin{enumerate}[(a)]
  \item By construction it holds that $\y_{P_i}\trt \overline \d_{P_i}=0$.
 Then it holds that for all $\mu >0$:
 $$\y_{P_i}\trt(\widehat \xbold^k_{P_i}+\mu \overline \d_{P_i}) = \y_{P_i}\trt \widehat \xbold^k_{P_i} = \y_{P_i}\trt \x_{P_i}^k,$$
 where last equality holds since $\widehat \xbold^k_{P_i} \in {\cal F}_{P_i}$.
Therefore we can conclude that for all $\mu > 0$:
$$\widehat \xbold^k_{P_i} + \mu \overline \d_{P_i}\notin [0,C]^{|P_i|},$$
and then either $\widehat x^k_{\bar i}$ or $\widehat x^k_{\bar j}$ must be on a bound.
In particular, supposing w.l.o.g. that is the component $\bar i$ the one on the bound, if $\overline d_{\bar i} > 0$ then $\widehat x^k_{\bar i} = C$ and then we can write
$$
0 < \bar\alpha \overline d_{\bar i} = \overline x^k_{\bar i} - x^k_{\bar i} \leq C - x^k_{\bar i} = \widehat x^k_{\bar i} - x^k_{\bar i};
$$
otherwise $\overline d_{\bar i} < 0$ then $\widehat x^k_{\bar i} = 0$ and then
$$
0 > \bar\alpha \overline d_{\bar i} = \overline x^k_{\bar i} - x^k_{\bar i} \ge 0 - x^k_{\bar i} = \widehat x^k_{\bar i} - x^k_{\bar i}.
$$
In both cases it holds that
$$
|\bar\alpha\overline d_{\bar i}| \leq |\widehat x^k_{\bar i} - x^k_{\bar i}|.
$$
Therefore noting that $|\bar\alpha \overline d_{\bar i}| = |\bar \alpha|$ and that $|\bar\alpha| \sqrt{2} =  \|\bar \xbold^k - \xbold^k\|$ we can conclude that
$$
 \|\widehat \xbold^k_{P_i} - \x_{P_i}^k\| \ge |\bar \alpha| =  \frac{1}{\sqrt{2}}  \|\overline \xbold^k - \xbold^k\|.
$$

\item Since $\overline \d_{P_i}$ is a feasible direction at $\widehat \xbold^k_{P_i}$ and since $\widehat \xbold^k_{P_i}$ is an optimal solution of problem \eqref{eq: problem i-th proximal} at $\xbold^k$, by \eqref{eq: dirder}, we can write
\begin{equation}\label{eq: cond optimal point}
\left[ \nabla_{P_i} f_{P_i} \left(\widehat \xbold^k_{P_i}, \xbold^k_{-P_i}\right) + \tau^k_i (\widehat \xbold^k_{P_i} - \x_{P_i}^k) \right]\trt \overline \d_{P_i} \geq 0.
\end{equation}
Since $\bar \xbold^k$ is a solution of \eqref{eq: prob SMO}, and being $-\overline \d$ a feasible direction for \eqref{eq: prob SMO} at $\bar \xbold^k$, then, by the minimum principle and since $(\bar i, \bar j) \subseteq P_i$, we can write
$$
\nabla_{P_i} f_{P_i} \left(\overline \x_{P_i}^k, \xbold^k_{-P_i}\right)\trt \overline \d_{P_i} \leq 0.
$$
And therefore by \eqref{eq: cond optimal point} we can write
\begin{align}\label{eq: cond gradients}
\left[ \nabla_{P_i} f_{P_i} \left(\widehat \xbold^k_{P_i}, \xbold^k_{-P_i}\right) + \tau^k_i (\widehat \xbold^k_{P_i} - \x_{P_i}^k) \right]\trt \overline \d_{P_i} \geq \nonumber\\
 \nabla_{P_i} f_{P_i} \left(\overline \x_{P_i}^k, \xbold^k_{-P_i}\right)\trt \overline \d_{P_i}.
\end{align}
By assumptions, $\sigma > 0$ exists such that
\begin{align}\label{eq: str monotonicity}
 \sigma \|\overline \x_{P_i}^k - \x_{P_i}^k \|^2 &\leq (\overline \x_{P_i}^k - \x_{P_i}^k)\trt Q_{P_i P_i} (\overline \x_{P_i}^k - \x_{P_i}^k) \nonumber\\
 &= \left[ \nabla_{P_i} f_{P_i} \left(\overline \x_{P_i}^k, \xbold^k_{-P_i}\right) - \nabla_{P_i} f_{P_i} \left(\x_{P_i}^k, \xbold^k_{-P_i}\right) \right]\trt \bar\alpha \overline \d_{P_i}.
\end{align}
Then combining \eqref{eq: cond gradients} and \eqref{eq: str monotonicity} we can write
\begin{align*}
 &\sigma \|\overline \x_{P_i}^k - \x_{P_i}^k \|^2 \leq \tau^k_i (\widehat \xbold^k_{P_i} - \x_{P_i}^k)\trt \bar\alpha \overline \d_{P_i} + \nonumber\\
 &\left[ \nabla_{P_i} f_{P_i} \left(\widehat \xbold^k_{P_i}, \xbold^k_{-P_i}\right) - \nabla_{P_i} f_{P_i} \left(\x_{P_i}^k, \xbold^k_{-P_i}\right) \right]\trt \bar\alpha \overline \d_{P_i} \leq \nonumber\\
 &(\tau^k_i + \|Q_{P_iP_i}\|) \|\widehat \xbold^k_{P_i} - \x_{P_i}^k\| \|\bar\alpha\overline \d_{P_i}\| = \nonumber\\
 &(\tau^k_i + \lambda_{\max}^{\Qbold}) \|\widehat \xbold^k_{P_i} - \x_{P_i}^k\| \|\overline \x_{P_i}^k - \x_{P_i}^k\|.
\end{align*}
Therefore we obtain
$$
\|\widehat \xbold^k_{P_i} - \x_{P_i}^k\| \geq
\frac{\sigma}{\tau^k_i + \lambda_{\max}^{\Qbold}} \|\overline \x_{P_i}^k - \x_{P_i}^k \| = \frac{\sigma}{\tau^k_i + \lambda_{\max}^{\Qbold}}  \|\overline \xbold^k - \xbold^k\|,
$$
and finally we have the proof.
\end{enumerate} \QEDA
\end{proof}
\noindent
\begin{theorem}\label{th: wss}
Assume that any principal submatrix of $Q$ of
 order $2$ is positive definite.
 Let $\xbold^k$ be a feasible point for problem \eqref{eq:problem} and let $(\bar i,\bar j)$ be a pair of indices such that $\bar i \in I_{up}(\xbold^k)$ and $\bar j \in I_{low}(\xbold^k)$.
 Suppose that a block $P_i \subseteq \{1, \ldots, n\}$ exists such that
 $(\bar i,\bar j) \subseteq P_i$.
Let $\bar \xbold^k$ be the unique solution of problem \eqref{eq: prob SMO} and suppose that $\widetilde \epsilon > 0$ exists such that
 \begin{equation}\label{eq: pair with mvp}
  \|\bar \xbold^k - \xbold^k\| \geq \widetilde \epsilon S_{\texttt{MVP}}(\xbold^k).
 \end{equation}
 Then $P_i$ is a descent block.
\end{theorem}
\begin{proof}
 By Lemma \ref{lemma: wss} we know that \eqref{eq: pair descent} holds. Therefore by combining \eqref{eq: pair descent} and \eqref{eq: pair with mvp} we obtain the proof. \QEDA
\end{proof}

Theorem \ref{th: wss} shows that we can build a descent block at the cost of computing a pair that satisfies \eqref{eq: pair with mvp}. Clearly the most violating pair does it, but it is easy to see that
any pair that ``sufficiently'' violates KKT conditions can be used as well.

Now we give a further theoretical result which guarantees that at each iteration of Algorithm \ref{al: svm decomposition} at least one descent block can be built.
\begin{theorem}\label{termination}
 Let $\xbold^k$ be a feasible, but not optimal, point for problem \eqref{eq:problem}
 then at least one descent block $P_i \subseteq\{1, \ldots, n\}$ exists.
\end{theorem}
\begin{proof}
By Proposition \ref{pr: M=0}, if $\xbold^k$ is not optimal then $I_{up}(\xbold^k) \neq \emptyset$, $I_{low}(\xbold^k) \neq \emptyset$ and $S_{\texttt{MVP}}(\xbold^k)>0$. Therefore $P_i = (i_\texttt{MVP},j_\texttt{MVP})$ is a descent block. \QEDA
\end{proof}

\section{Global Convergence in a Realistic Setting}\label{sec:realistic}

So far we have proved that, under some suitable conditions, Algorithm \ref{al: svm decomposition} either converges in a finite number of iterations to a solution of problem \eqref{eq:problem} or the produced sequence $\{\x^k\}$  satisfies \eqref{eq: diminishing lim convergence}.
However the fact that $S_{\texttt{MVP}}(\xbold^k)$ goes to zero is not enough to guarantee asymptotic convergence of Algorithm \ref{al: svm decomposition} to a solution of problem \eqref{eq:problem}.
Indeed,  this is due to the discontinuous nature of indices sets $I_{up}$ and $I_{low}$ that enters the definition of $S_{\texttt{MVP}}(\xbold^k)$.
Actually, this is a well known theoretical issue in decomposition methods for the SVM training problem. However, even in the case when the algorithm were proved to asymptotically converge to an optimal solution, the validity of a stopping criterion based on the KKT conditions must be verified \cite{lin:stopping}.
A possible way to sorting out these theoretical issues is to use some theoretical tricks. For example by properly inserting some standard MVP iterations in the produced sequence $\{x^k\}$ \cite{hybrid} or by dealing with $\epsilon-$solutions \cite{keerthigilbert:02}.
All these theoretical efforts can be encompassed in a realistic numerical setting.
Indeed all the papers discussing about decomposition methods rely on the fact that
the indices sets $I_{up}$ and $I_{low}$ can be computed in exact arithmetic.
In practice what it can actually be computed are the following $\epsilon-$perturbations of sets $I_{up}$ and $I_{low}$
\begin{equation*}
  I^{\epsilon}_{up}(\xbold) := \{r \subseteq \{1,\dots,n\} : \ x_r\leq C-\epsilon, \ y_r=1,\mbox{ or } x_r\geq \epsilon, \ y_r=-1 \},
\end{equation*}
\begin{equation*}
 I^{\epsilon}_{low}(\xbold) := \{r \subseteq \{1,\dots,n\} : \ x_r\leq C-\epsilon, \ y_r=-1,\mbox{ or } x_r\geq \epsilon, \ y_r=1 \},
\end{equation*}
with  $\epsilon>0$.
Consequently we can define
at a feasible point $\x$
the following quantities
$$ m^{\epsilon}(\xbold)=\max_{r \in  I^{\epsilon}_{up}(\xbold)} -\frac{\nabla f(\xbold)_r}{y_r}, \qquad  M^{\epsilon}(\xbold)= \min_{r \in I^{\epsilon}_{low}(\xbold)} -\frac{\nabla f(\xbold)_r}{y_r}.
$$
As a matter-of-fact an effective optimality condition which can be used is
\begin{equation}\label{optimality real}
 m^{\epsilon}(\x^k) \leq M^{\epsilon}(\x^k) + \eta,
\end{equation}
where $\eta > 0$ is a given tolerance.
Note that any asymptotically convergent decomposition algorithm can actually converge only to a point satisfying \eqref{optimality real}, rather than \eqref{optimality}.

It is easy to see that
$  I^{\epsilon}_{up}(\xbold)\subseteq   I_{up}(\xbold)$
and $  I^{\epsilon}_{low}(\xbold)\subseteq   I_{low}(\xbold)$ for all $\epsilon>0$. Furthemore
in \cite{phdtesi:risi}, it has been proved the following result.

\begin{proposition}\label{epsact}
Let $\{x^k\}$ be a  sequence of feasible points converging to a
point $\bar x \in {\cal F}$. Then, there exists a scalar $\bar
\epsilon >0$ (depending only on $\bar x$) such that for every
$\epsilon\in (0,\bar \epsilon]$ there exists an index
$\bar k=k_{\epsilon}$ for which
$$
I^{\epsilon}_{up}(\xbold^k)\equiv I_{up}(\xbold^k) \quad\quad {\rm and} \quad\quad
I^{\epsilon}_{low}(\xbold^k)\equiv  I_{low}(\xbold^k) \quad\quad \mbox{for~all}~k\ge
\bar k.
$$
\end{proposition}
This proposition allows to state that for $k$ sufficiently large and $\epsilon$ sufficiently small
using index sets $ I^{\epsilon}_{up}(\xbold)$ and $ I^{\epsilon}_{low}(\xbold)$ is equivalent to use the exact ones $I_{up}$ and $I_{low}$ and
 we have that
$$ m^{\epsilon}(\xbold)= m(\xbold)\qquad \mbox{and}\qquad M^{\epsilon}(\xbold)= M(\xbold),$$
so that, for any $\epsilon\in (0,\bar \epsilon]$ and $k\ge \bar k$,
\eqref{optimality real} reduces to the concept of $\eta$-optimal solution introduced in \cite{keerthigilbert:02}.
However this is not true far from a solution and/or for a wrong value of $\epsilon$, being $\bar \epsilon$ unknown. Reducing $\epsilon$ to the machine precision $\epsilon_{\rm mach}$ is the best that we can do in a numerical implementation, so that one can argue that for $\epsilon = \epsilon_{\rm mach}$ if $ I^{\epsilon}_{up}(\xbold)=\emptyset$ or $ I^{\epsilon}_{low}(\xbold)=\emptyset$, a solution has been reached within the possible tolerance.

Given a point $\x^k \in \cal F$, we consider the MVP $\epsilon-$step $ S^{\epsilon}_{\texttt{MVP}}(\xbold^k)$ obtained by using
$I^{\epsilon}_{up}(\xbold^k)$ and $I^{\epsilon}_{low}(\xbold^k)$  instead of $I_{up}(\x^k)$ and $I_{low}(\x^k)$. As  a consequence of the definition itself, for any MVP $\epsilon-$direction  $\d^k_{\texttt{MVP},\epsilon}$ we get that the feasible $\epsilon-$stepsize $\bar\beta_{\epsilon}^k$ defined as in \eqref{eq:t} remains bounded from zero by $\epsilon$.

It is easy to see that all results stated so far for Algorithm \ref{al: svm decomposition} are still valid if we consider the $\epsilon-$definition $S^{\epsilon}_{\texttt{MVP}}(\xbold^k)$ rather than $S_{\texttt{MVP}}(\xbold^k)$. Furthemore we have the following result, that fill the gap of convergence.
\begin{theorem}\label{th:numerical convergence}
 Let $\epsilon>0$ and $\eta>0$ be given.
 Let $\{\x^k\}$ be a sequence of feasible points such that $ I^{\epsilon}_{up}(\x^k) \neq \emptyset$, $ I^{\epsilon}_{low}(\x^k) \neq \emptyset$ and $$\lim_{k \to \infty}  S^{\epsilon}_{\texttt{MVP}}(\xbold^k)=0.$$ Then $\bar k>0$ exists such that, for all $k\ge \bar k$, $\x^k$ satisfies \eqref{optimality real}.
\end{theorem}
\begin{proof}
By definition of $S^{\epsilon}_{\texttt{MVP}}(\xbold^k)$ we get
 \begin{equation}\label{eq: bar alpha to 0}
  0=\lim_{k \to \infty}  S^{\epsilon}_{\texttt{MVP}}(\xbold^k)=\sqrt{2}\lim_{k \to \infty} |\alpha^{k}_{\texttt{MVP},\epsilon}|.
 \end{equation}
 Since by construction $\bar\beta_{\epsilon}^k\ge \epsilon$, by \eqref{eq:alfa_ott}
  we get that \eqref{eq: bar alpha to 0} implies that $\bar k>0$ exists such that, for all $k\ge\bar k$, we have $-\nabla \fbold(\x^k)\trt \d_{\texttt{MVP},\epsilon}^k \le \eta$, which implies \eqref{optimality real}. \QEDA
\end{proof}

\section{Practical Algorithmic Realizations}\label{sec:implementation}
Algorithm \ref{al: svm decomposition} includes a vast
amount of specific strategies that may vary according to several implementation
choices. Various alternative may be related to
the blocks dimension, the blocks composition, the blocks selection, the way
to enforce convergence conditions and the methods used to solve the subproblems. Different algorithms can be designed exploiting these degrees of freedom.
In this section we discuss about some possible alternatives and we
suggest some practical implementations of Algorithm \ref{al: svm decomposition}.

The dimension of the blocks is a key factor
for the training performances.
It influences the way in which the subproblems
can be solved so it should be carefully determined according
to the dataset nature.
Mainly we can consider two opposite strategies:
blocks of minimal dimension (i.e. SMO-type methods) and higher dimensional blocks.
In SMO-type methods we can take advantage of the fact that, for each block, the subproblem can be solved
analytically. On the other
hand each SMO-block may yield a small decrease of the objective
function and slow identification of the support vectors, so that to get fast convergence the simultaneous optimization
of a great number of SMO-blocks may be needed
when dealing with high dimensional instances.
This choice could be well suited for an architecture composed of a great amount
of simple processing units, like that of the recent Graphic
Processing Unit (GPU).
Higher dimensional blocks require the solution of the subproblems by means of some optimization procedure hence the solution of each block may need greater time consumption. On the other hand the decrease of the objective function and the identification of the support vectors may be faster so that less iterations should be needed.
Hence this choice may be suitable  whenever
powerful but, generally, not many processing units are available.
Algorithm \ref{al: svm decomposition} encompasses also the possibility of
considering huge blocks assigned at each processor and using a decomposition method to solve the corresponding huge subproblems. This strategy essentially consists in iteratively splitting the original SVM instance into smaller SVMs, distributing them to available parallel processes and then gathering their solutions points in order to properly define the new iterate.

An efficient rule to partition the variables into blocks is
crucial for the rate of convergence of the algorithm.
A lot of heurisitc methods (see e.g. \cite{SVMlight,SVMlightconvergence,Zoutendijk}) have been studied to
obtain efficient rules for constructing subproblems that
can guarantee a fast decrease of the objective function (e.g. by determining the most violating pair).
Such methods usaully make use of first order information, implying that
the gradient should be partially or entirely updated at each iteration, and
this could be overwhelming in a huge dimensional framework.
However practical implementations with a (partially) random
composition of the blocks could be considered.

Once we have determined a blocks partition of the whole set of variables,
only a subset of the resulting subproblems may, in general, be
involved in the optimization process. Indeed, we may further restrict
the blocks used to update the current iterate by determining a subset ${\cal J}^k$.
We need to keep in mind that
the main computational burden is due to the gradient update.
Indeed at each iteration, the gradient update must be performed by computing
 the columns of $\Qbold$ related only to those variables that are chosen to be in the selected blocks ${\cal J}^k$.
 Hence the choice of  ${\cal J}^k$ may take into account both the decrease of the objective function and the computational effort for updating the gradient.
A minimal threshold on the percentage of objective function decrease
associated to each block, with respect to the cumulative decrease of
every block, could be a possible discriminant for a blocks selection rule
in order to avoid useless computations.

The practical effectiveness of the algorithm is highly related to the way to enforce convergence conditions stated so far.
If, for example, we do not take into account the rule of taking at least one descent block every $L$ iterations, we could take at each iteration the same partition of variables. This short-sighted strategy would be totally ineffective, since, in this case, the algorithm would lead only to an equilibrium of the generalized Nash equilibrium problem (see e.g. \cite{fks,fk,fs}) in which the players solve the fixed subproblems, but not to a solution of the original SVM, see the following example.
\begin{example}
Let us consider a dual problem with four variables and in which $Q=I$, $\ybold = (1 \; 1 \; -1 \; -1)\trt$ and $C=1$. The unique solution of this problem is $\xbold^* = (1 \; 1 \; 1 \; 1)\trt$. Let $\xbold^k = (0 \; 0 \; 0 \; 0)\trt$ and suppose to consider a partition of two blocks: ${\cal P} = \{(1,2),(3,4)\}$. Then the best responses are $\widehat \xbold^k_{(1,2)} = (0 \; 0)\trt$ and $\widehat \xbold^k_{(3,4)} = (0 \; 0)\trt$, but this implies that $\xbold^{k+1} = \xbold^{k}$. Therefore if we do not modify ${\cal P}$ we will never move from the origin, and then will never converge to $\xbold^*$. However, being the origin a fixed point for the best responses of the two processes, it is, by definition, a Nash equilibrium for the game involving these two players.
\end{example}

Regarding the choice of the steplenght $\alpha^k$ we can use two different strategies: a linesearch or a diminishing stepsize rule. In the first case, as showed before, we can use the exact minimization formula \eqref{eq: exact}. In the second case a simple rule  could be
$\alpha^{k}=\frac{1}{k^\xi}$,
with $\xi \in (0,1]$, but different choices are also possible.
Although preliminary tests
showed that the exact minimization is more effective than any other choice, the diminishing
stepsize strategy, besides being easy to implement, requires much less computations and this could be of great practical interest
for high dimensional instances (training set with many samples and many dense features).

As mentioned above, when the dimension of the blocks
is more than two, an optimization algorithm is needed
to obtain a solution of the subproblems.
Due to its particular structure (convex quadratic
objective function over a polyhedron), a lot
of exact or approximate methods can be applied for the
solution of problem \eqref{eq: problem i-th proximal}. We simply point out that, whenever the dimension
of the blocks is so big that each block can be considered
a sort of smaller SVM, a slight modification of any efficient software
for the sequential training of SVMs (like LIBSVM)
could be used to perform a single optimization step.

As a matter of example, we propose a SMO-type parallel scheme derived from Algorithm \ref{al: svm decomposition} for which we developed two matlab prototypes.
This realization, that we call PARSMO, is based on using a partition ${\cal P}$ of minimal dimension blocks thus performing multiple SMO steps simultaneously in order to build the search direction $\d^k$. Each SMO step is assigned to a parallel process that analytically solves a two-dimensional subproblem. Thus the computational effort of each processor is very light and communications must be very fast thus being suitable for a multicore environment.
\begin{center}
\framebox{\setlength{\miniwidthb}{\textwidth}
    \addtolength{\miniwidthb}{-4pc}
    \begin{minipage}{\miniwidthb}
 \begin{algorithm}\label{al: parallel smo}{\rm
     \begin{center}
        {\bf PARSMO}
      \end{center}

\begin{description}
      \item[{\bf Initialization}]  Set $\xbold^0 = 0$, $\nabla \fbold^0 = -\ebold$, $q \ge 1$, $\epsilon>0$, $\eta>0$ and $k=0$.\\
Select
$$(i_1,j_1)\in I_{up}^{\epsilon,MVP}(\xbold^k)\times
I_{low}^{\epsilon,MVP}(\xbold^k).$$

\item[{\bf Do while}]$\left(-{\nabla\fbold^k_{i_1}}{y_{i_1}}+{\nabla\fbold^k_{j_1}}{y_{j_1}} \ge \eta\right)$
\begin{description}
 \item[\bf{S.1} (Blocks definition)]\
\par\smallskip\noindent
 Choose $(q-1)$ pairs $\{(i_2,j_2),(i_3,j_3),\dots,(i_q,i_q)\}$.\\
 Set ${\cal J}^k = \{(i_1,j_1),(i_2,j_2),\dots,(i_q,i_q)\}$.
 \item[\bf{S.2} (Parallel computation)]\
 \par\smallskip\noindent
 For each pair $(i_h,j_h) \in {\cal J}^k$ compute in parallel:
 \begin{enumerate}
  \item kernel columns $\Qbold_{* i_h}$ and $\Qbold_{* j_h}$ (if not available in the cache),
   \item $t_h \d^{(i_hj_h)}$ with $\d^{(i_hj_h)}$ defined as in \eqref{smo_generic} and $t_h$ as in \eqref{eq:alfa_smo}.
\end{enumerate}
\par\smallskip\noindent
\item [\bf{S.3} (Direction)] $\displaystyle \d^k=\sum_{(i_h,j_h)\in {\cal J}^k } t_h\d^{(i_hj_h)}$
\par\smallskip\noindent
\item[\bf{S.4} (Stepsize)]
Compute the steplenght $\alpha^k$ as in \eqref{eq: exact}.
 \item[\bf{S.5}  (Update)]\
\par\smallskip\noindent
Set $\xbold^{k+1} = \xbold^k + \alpha^k \dbold^k$.\\
Set $\displaystyle\nabla \fbold^{k+1} = \nabla \fbold^k + \alpha^k\sum_{h=1}^q t_h \left(d^k_{i_h} \Qbold_{* i_h}  + d^k_{j_h}\Qbold_{* j_h} \right)$.\\
Set $k = k + 1$.\\
Select
$$(i_1,j_1)\in I_{up}^{\epsilon,MVP}(\xbold^k)\times
I_{low}^{\epsilon,MVP}(\xbold^k).$$
\end{description}
\item[{\bf End While}]
\item[{\bf Return}] $\x^k$.
\end{description}}\end{algorithm}
\end{minipage}}
\end{center}

\noindent
Starting from the feasible null vector $\xbold^0 = 0$, which is a well known suitable choice for SVM training algorithms beacuse allows to  initialize the gradient  $\nabla \fbold^0$ to $-\ebold$, the algorithm selects at each iteration the  most violating pair $(i_1,j_1)\in  I^{\epsilon,MVP}_{up}(\xbold^k) \times I^{\epsilon,MVP}_{low}(\xbold^k)$ and further $(q-1)$ pairs that all together make up ${\cal J}^k$.
Search direction $\dbold^k$ at {\bf S.3} is obtained by analytically computing stepsise $t_h$ for  the $q$ subproblems of type \eqref{eq: problem i-th proximal} related to the pairs in ${\cal J}^k$. Direction $\d^k$ is simply computed by summing up all the SMO steps; it has $\|\d^k\|_0=2q$ with $q\ge 1$ which depends on the number of parallel processes that we want to activate.
Finally at Step 4 steplenght $\alpha^k$ that exactly minimizes the objective function along $\dbold^k$ is obtained by \eqref{eq: exact}; note that this step requires no further kernel evaluations. The same holds for the gradient update.
Of course as standard in SVM algorithms, a caching strategy can be exploited to limit the computational burden due to the evaluation of kernel columns.

In PARSMO it remains to specify how to select pairs forming blocks ${\cal J}^k$. Since the most expensive computational burden is due to the calculation of kernel columns, we want to analyse the impact of a massive use of the cache with respect to a standard one.
Indeed we propose two different matlab implementations of PARSMO that use a cache strategy in two different ways. We would compare the performance with a standard sequential MVP implementation in order to analyze possible advantages of the PARSMO scheme.


In the first implementation, the $q-1$ pairs, in addition to a MVP, are selected by choosing those pairs that most violate the first order optimality condition, like in the SVM$^{\mbox{light}}$ algorithm \cite{SVMlight}.
Hence
we select $q$ pairs $(i_h, j_h)\in I^{\epsilon}_{up}(\xbold^k)\times I^{\epsilon}_{low}(\xbold^k)$
 sequentially so
that
$$- y_{i_1} \nabla f(\xbold^k)_{i_1}\ge - y_{i_2}\nabla
f(\xbold^k)_{i_2}\ge \cdots \ge -y_{i_q}\nabla
f(\xbold^k)_{i_q},$$
and
$$-y_{j_1}\nabla f(\alpha^k)_{j_1}\le -y_{j_2}\nabla
f(\xbold^k)_{j_2}\le \cdots \le -y_{j_q}\nabla
f(\xbold^k)_{j_q}.$$
In this case, although we can use a standard caching strategy, we cannot control the number of kernel columns  evaluations at each iteration that in the worst case can be up to $2q$. The computation of kernel columns $\Qbold_{* i_h}$ and $\Qbold_{* j_h}, \forall h \in \{1,\dots,q\},$  can be performed in parallel by the processors empowered to solve the subproblems. In this case the number of kernel evaluation per iteration would be of course greater than those of a standard MVP, but the overall number of iterations may decrease. Thus we keep the advantages of performing simple analytic optimization, as in SMO methods, whilst moving $2q$ components at the time, as in SVM$^{light}$. We note that  reconstruction of the overall gradient $\nabla \fbold^{k+1}$ can be parallelized among the $q$ processors and requires a synchronization step to take into account stepsize $\alpha^k$. Thus
the CPU-time needed is essentially equivalent to a gradient update of a single SMO step. In this approach the transmission time among processors may be quite significant and this strictly depends on the parallel architecture.
We refer to this implementation as PARSMO-1.

The second implementation selects the $q-1$ pairs, in addition to a MVP $(i_1, j_1)\in I^{\epsilon}_{up}(\xbold^k)\times I^{\epsilon}_{low}(\xbold^k)$, exclusively among the indices of the columns currently available in the cache. To be more precise, let $\cal C$ be the index set of the kernel columns available in the cache. The $q-1$ pairs
 $(i_h, j_h)$  in ${\cal J}^k$  are selected following the same SVM$^{light}$ rule described above for PARSMO-1, but restricted to the index sets ${\cal C}\cap I^{\epsilon}_{up}(\xbold^k)\times {\cal C}\cap I^{\epsilon}_{low}(\xbold^k)$.
 We refer to it as PARSMO-2.
In this case the number of kernel evaluation per iteration is at most two as in a standard MVP implementation. The rationale of this version is to improve the performances of a classical MVP algorithm by using simultaneous multiple SMO optimizations without increasing the amount of kernel evaluations.

In order to have a flavour of the potentiality of these two parallel strategies, we
 performed some simple matlab experiments for the two versions PARSMO-1 and PARSMO-2.
All experiments have been carried out on a 64-bit intel-Core i7 CPU 870 2.93Ghz $\times$ 8.
Both PARSMO-1 and PARSMO-2 make use of a standard caching strategy, see \cite{LIBSVM}, with a cache memory of $500$ columns.
We perform experiments with $q= 1, 2, 4$ and $8$ parallel processes. Clearly the case  with $q=1$  corresponds to a classical MVP algorithm with a standard caching strategy.
It is worth noting that to preserve the good numerical behavior of PARSMO-1 and PARSMO-2, it is necessary the use of the ``gathering'' steplenght $\alpha^k$. In fact, further tests, not reported here, showed that by removing the use of $\alpha^k$ oscillatory and divergence phenomena may occur when using multiple parallel processes. This enforce the practical relevance of our theoretical analysis.

The major aim of the experiments in this preliminary contest is to highlight the benefits of simultaneously moving along multiple SMO directions.
We tested both PARSMO-1 and PARSMO-2 on six benchmark problems available at the LIBSVM site \url{http://www.csie.ntu.edu.tw/~cjlin/libsvmtools/datasets/}, using a standard setting for the parameters ($C=1$,  gaussian parameter $\gamma= 1/\#\mbox{features}$), see Table \ref{tab:dataset}.
\begin{table}[tth]
  \centering
  \begin{tabular}{l|rr|c}
    name & \#features & \#training data& kernel type\\
    \hline
    a9a & 123 & 32561 & gaussian\\
    gisette scale & 5000&6000&linear\\
    cod-rna & 8& 59535&gaussian\\
    real-sim &20958 &72309&linear\\
    rcv1 &47236 &20242&linear\\
    w8a &300& 49749&linear\\
    \hline
  \end{tabular}
  \caption{\label{tab:dataset} Training problems description.}
\end{table}

To evaluate the behavior of the algorithms we consider
the ``relative error'' ($RE$) as
$$ RE=\frac{|f^*-f|}{|f^*|},$$
where  $f^*$ is  the optimal known value of the objective function.
As regards PARSMO-1, for each problem we plot the RE versus
\begin{description}
\item{i)} the number of iterations (see Figure \ref{ReVSI});
\item{ii)}
the number of kernel evaluations per process, which is obtained by dividing the total number of kernel evaluations by the number of parallel processes involved (see Figure \ref{ReVSK}).
\end{description}

 \begin{figure}[!ht]
 \hspace{1.6em}
  \begin{minipage}[c]{.3\textwidth}
    \includegraphics[scale=0.2]{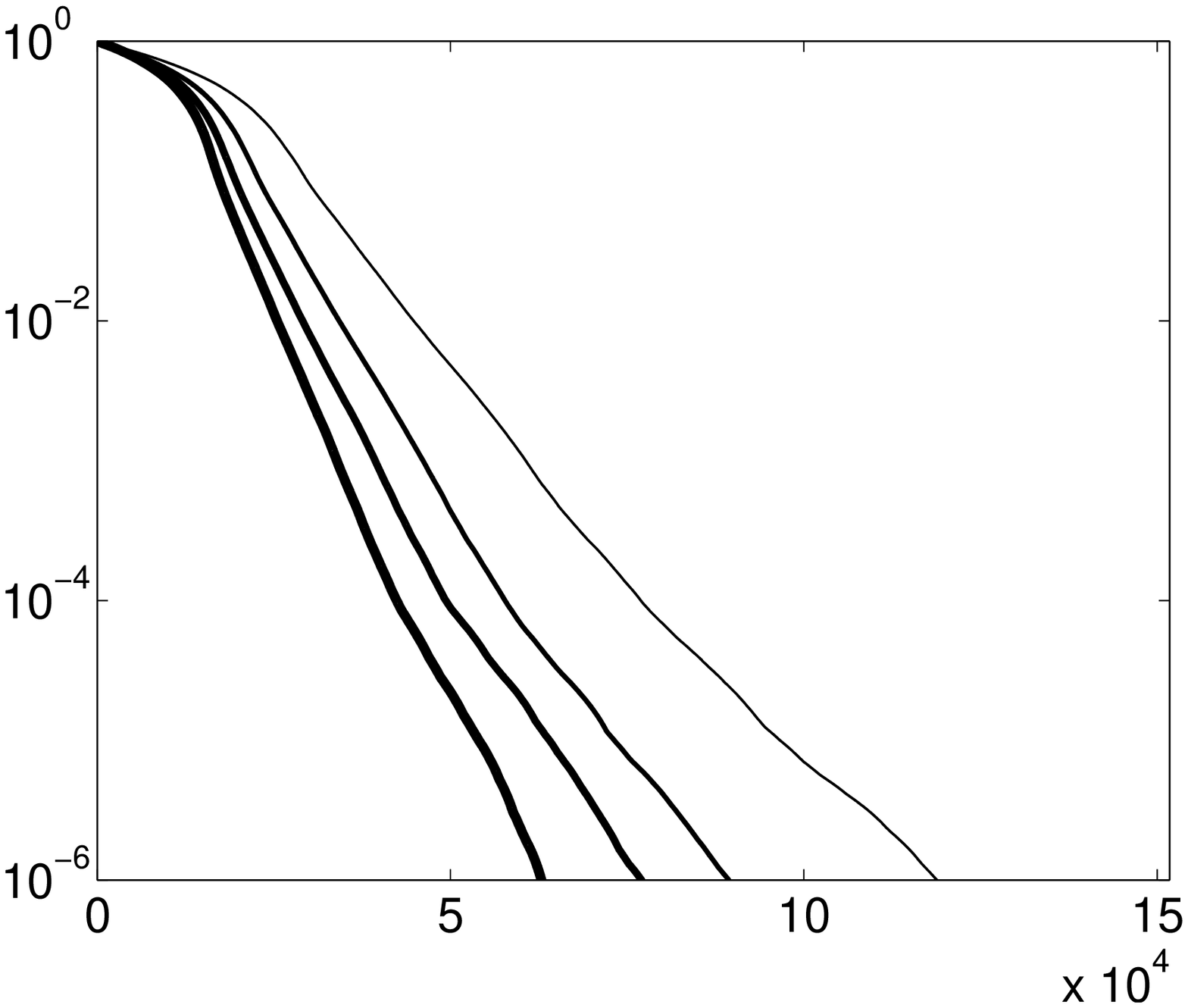}
  \end{minipage}
  \begin{minipage}[c]{.3\textwidth}
    \includegraphics[scale=0.2]{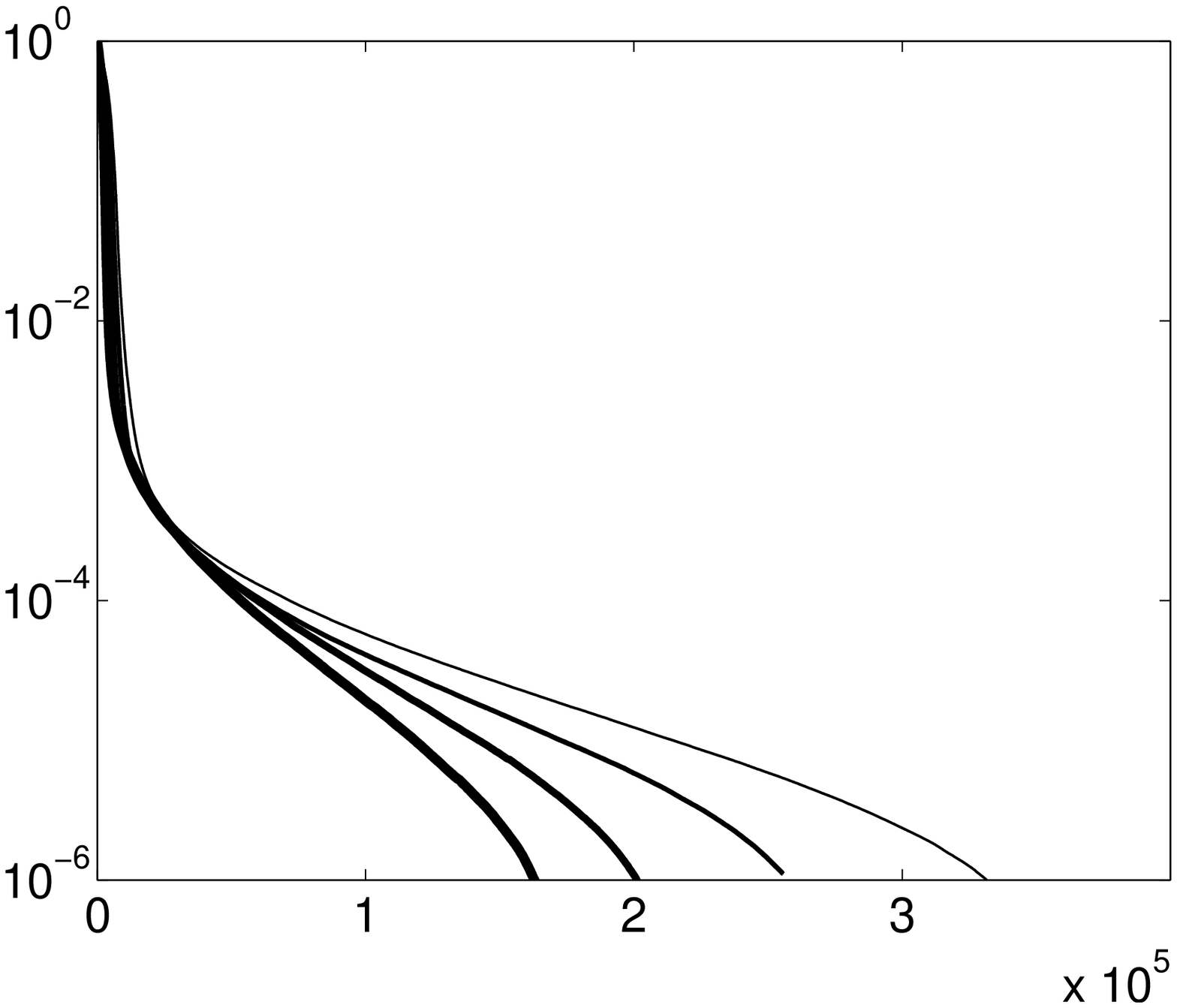}
  \end{minipage}
  \begin{minipage}[c]{.3\textwidth}
    \vspace{-1.6em}
    \includegraphics[scale=0.2]{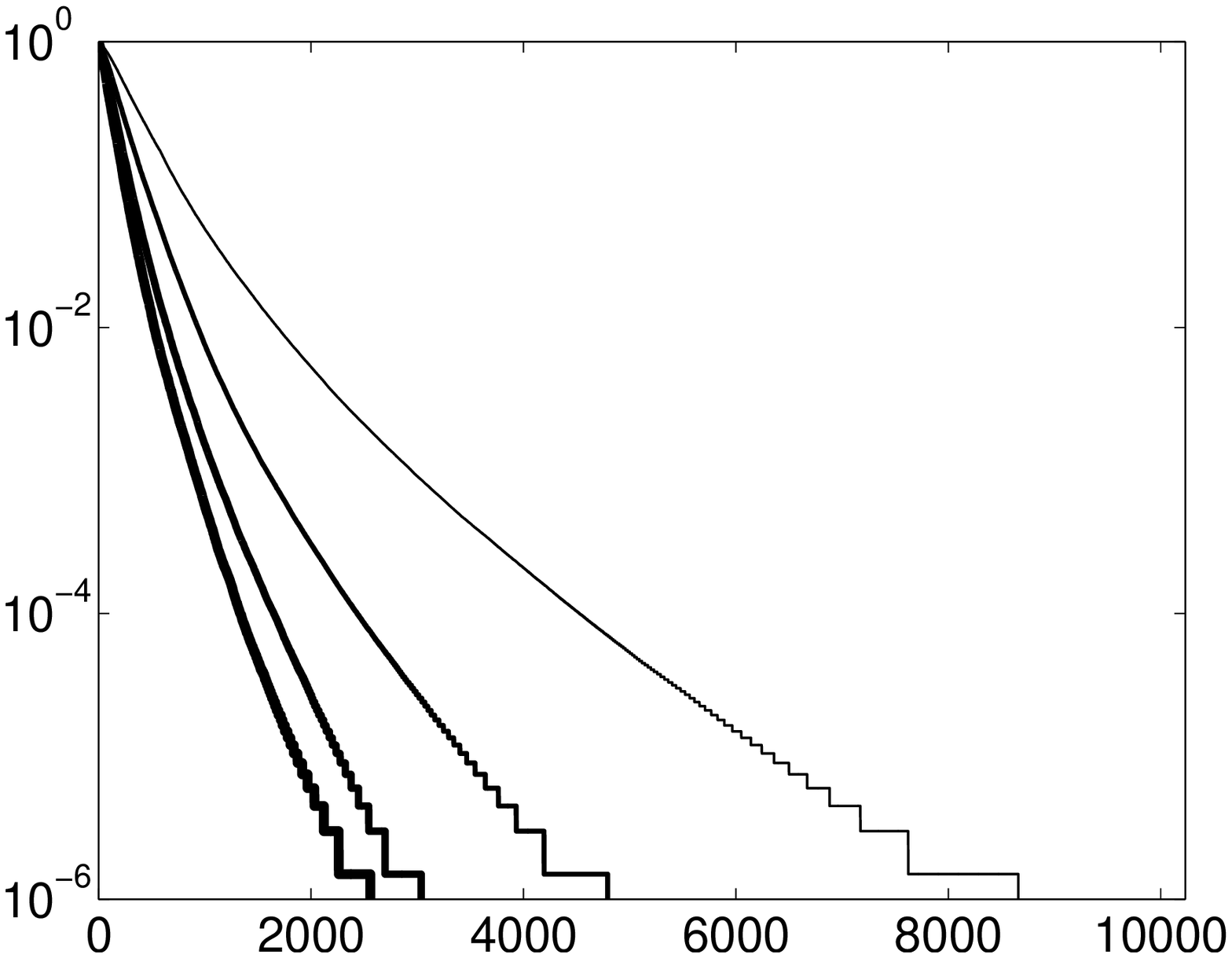}
  \end{minipage}
  \begin{minipage}[c]{.3\textwidth}
  \vspace{1.2cm}
  \begin{center}
   $\qquad$a9a
  \end{center}
  \end{minipage}
  \begin{minipage}[c]{.3\textwidth}
  \vspace{1.2cm}
  \begin{center}
   cod-rna
  \end{center}
  \end{minipage}
  \begin{minipage}[c]{.3\textwidth}
  \vspace{1.25cm}
   $\qquad$gisette scale
    \end{minipage}
 \vspace{-0cm}
 \hspace*{1.6em}
  \begin{minipage}[c]{.3\textwidth}
    \vspace{-2em}
    \includegraphics[scale=0.2]{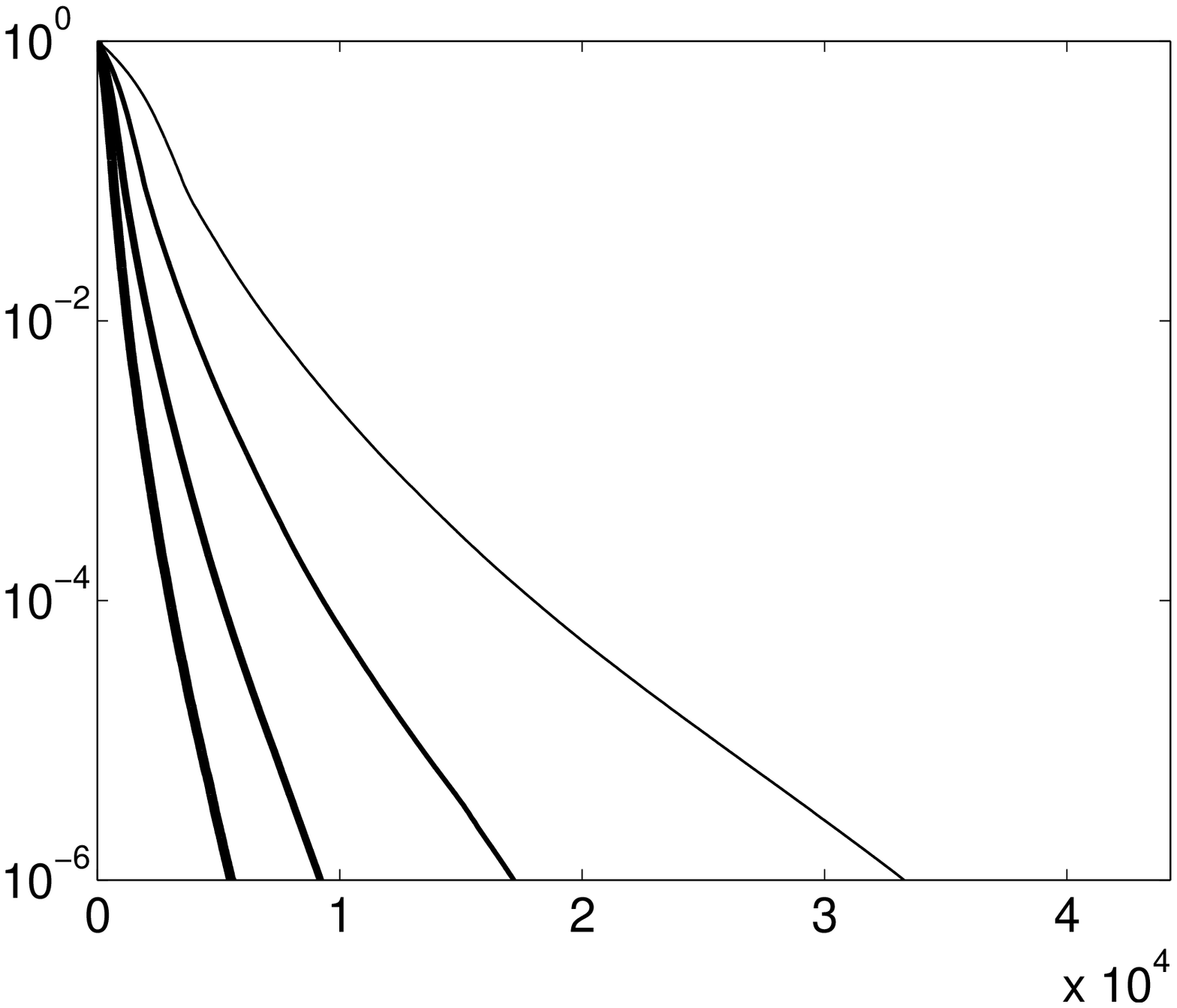}
  \end{minipage}
  \begin{minipage}[c]{.3\textwidth}
    \vspace{-2em}
    \includegraphics[scale=0.2]{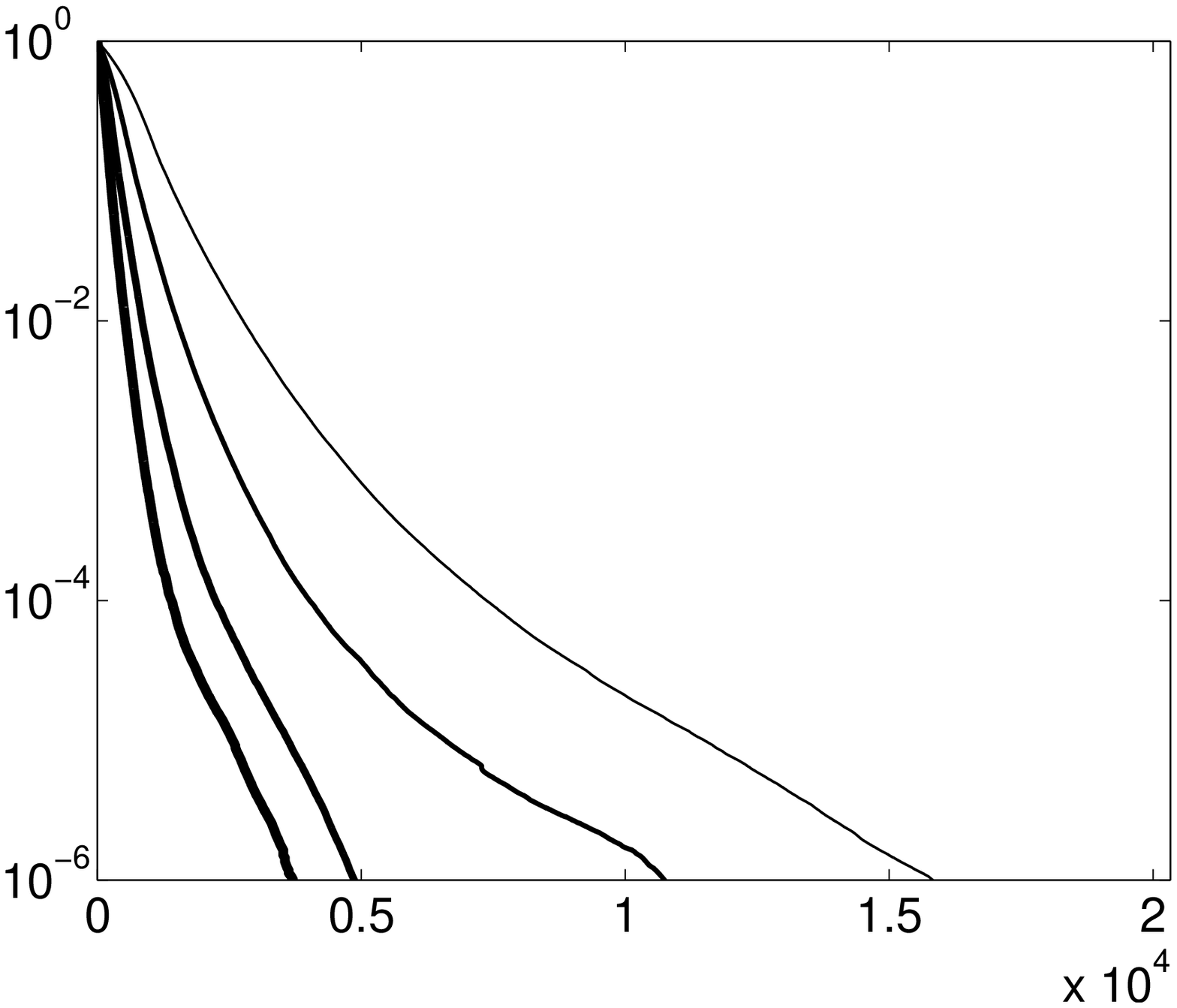}
  \end{minipage}
  \begin{minipage}[c]{.3\textwidth}
    \vspace{-2em}
    \includegraphics[scale=0.2]{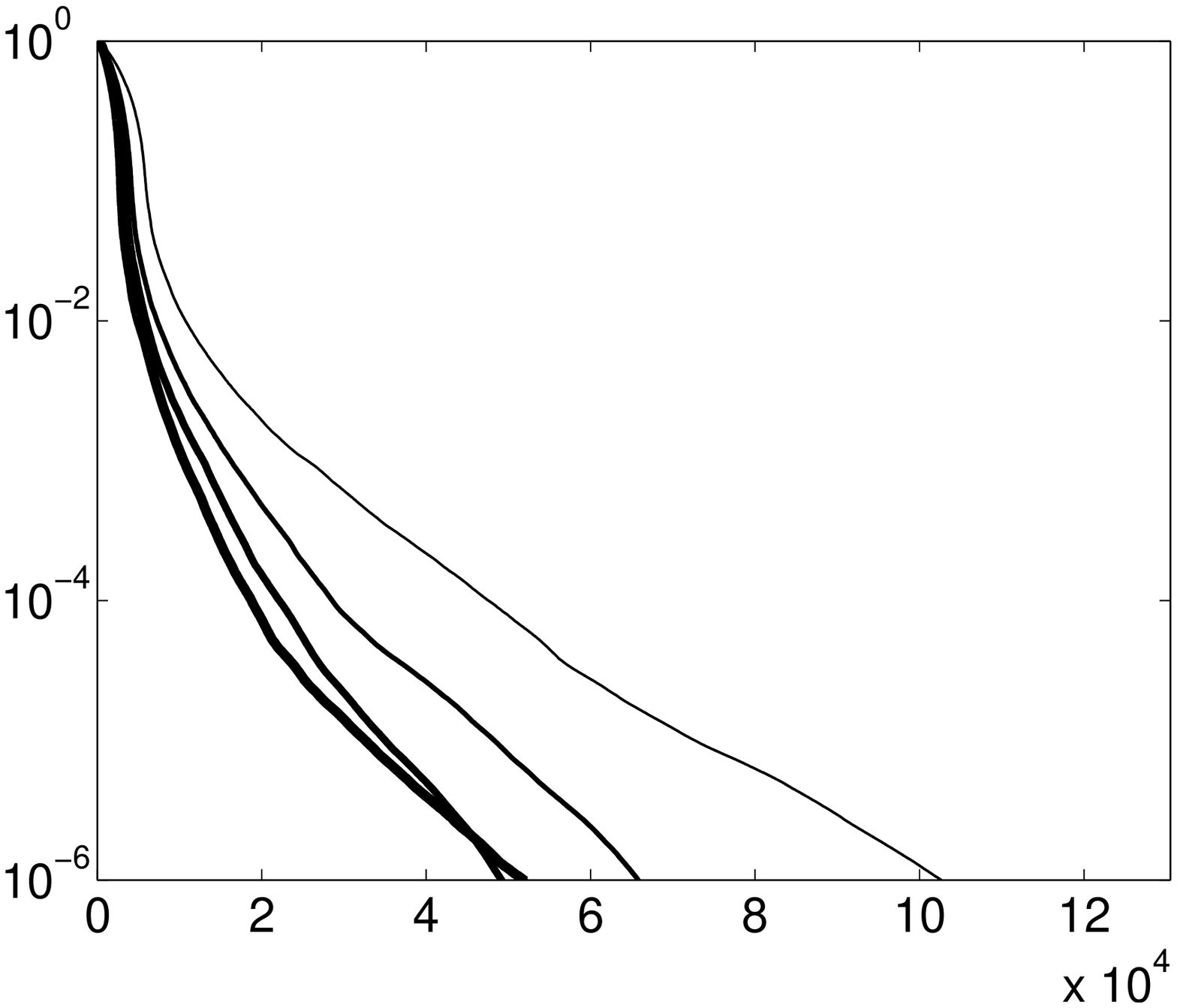}
  \end{minipage}
  \begin{minipage}[c]{.3\textwidth}
  \vspace{1.2cm}
  \begin{center}
   $\qquad$real-sim
  \end{center}
  \end{minipage}
  \begin{minipage}[c]{.3\textwidth}
  \vspace{1.2cm}
  \begin{center}
   rcv1
  \end{center}
  \end{minipage}
  \begin{minipage}[c]{.3\textwidth}
  \vspace{1.25cm}
   $\qquad\quad$w8a
    \end{minipage}
  \begin{minipage}[c]{1\textwidth}
    \vspace{2em}
	    \setlength{\unitlength}{0.4cm}
\begin{center}
 	    \begin{picture}(28.5,1)(0,0)
	    \put(0,0){\framebox(28.5,1)}
	    \put(2,0.4){\linethickness{0.2mm}\line(1,0){2}}
	    \put(4.5,0.3){$q=1$}
	    \put(9,0.4){\linethickness{0.35mm}\line(1,0){2}}
	     \put(11.5,0.3){$q=2$}
	    \put(16,0.4){\linethickness{0.50mm}\line(1,0){2}}
	    \put(18.5,0.3){$q=4$}
	    \put(23,0.4){\linethickness{0.65mm}\line(1,0){2}}
	    \put(25.5,0.3){$q=8$}
	    \end{picture}
	    \end{center}
  \end{minipage}
  \caption{PARSMO-1: Relative Error versus iterations. \label{ReVSI}}
\end{figure}
\begin{figure}[!ht]
 \hspace{1.6em}
  \begin{minipage}[c]{.3\textwidth}
    \includegraphics[scale=0.2]{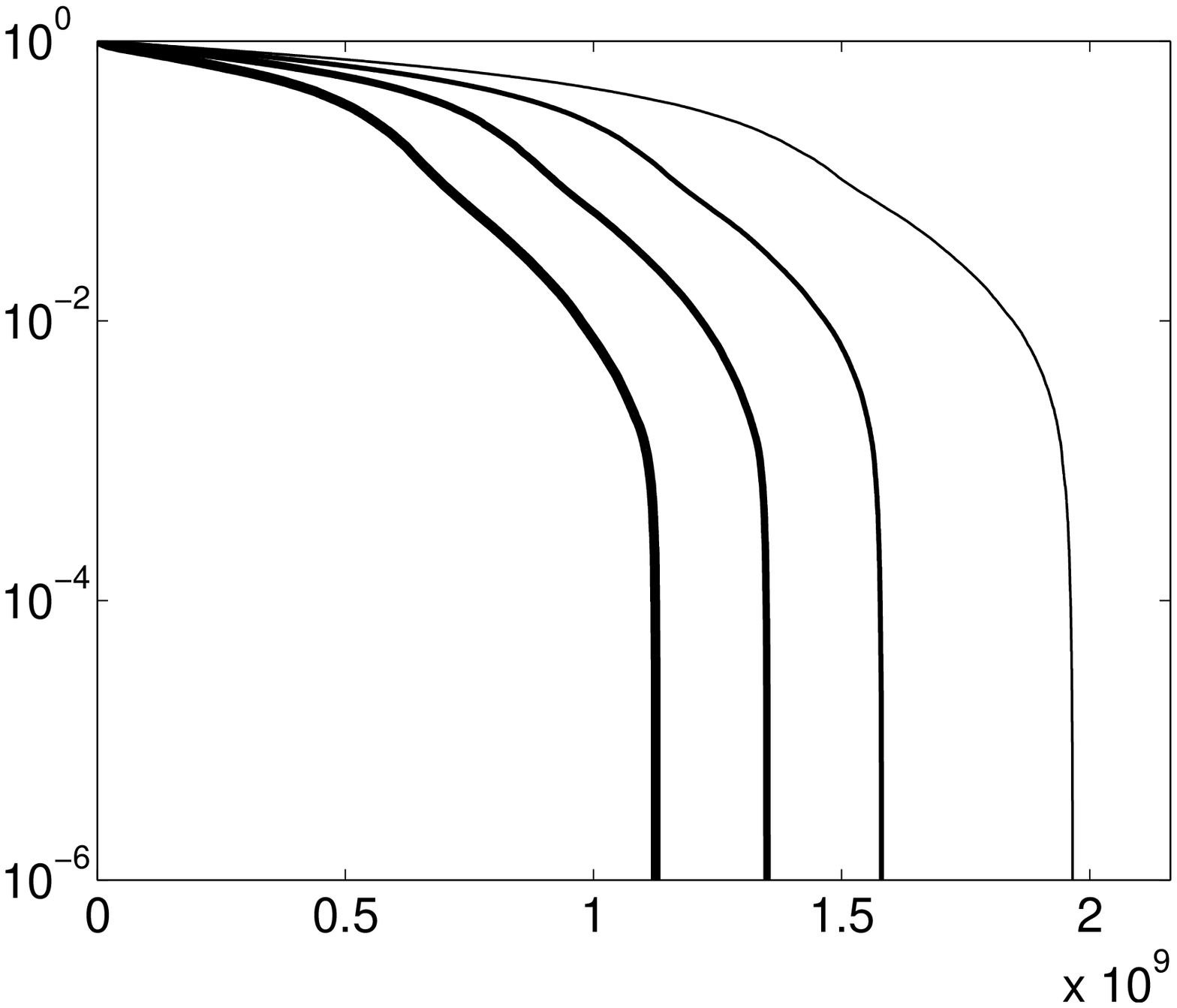}
  \end{minipage}
  \begin{minipage}[c]{.3\textwidth}
    \includegraphics[scale=0.2]{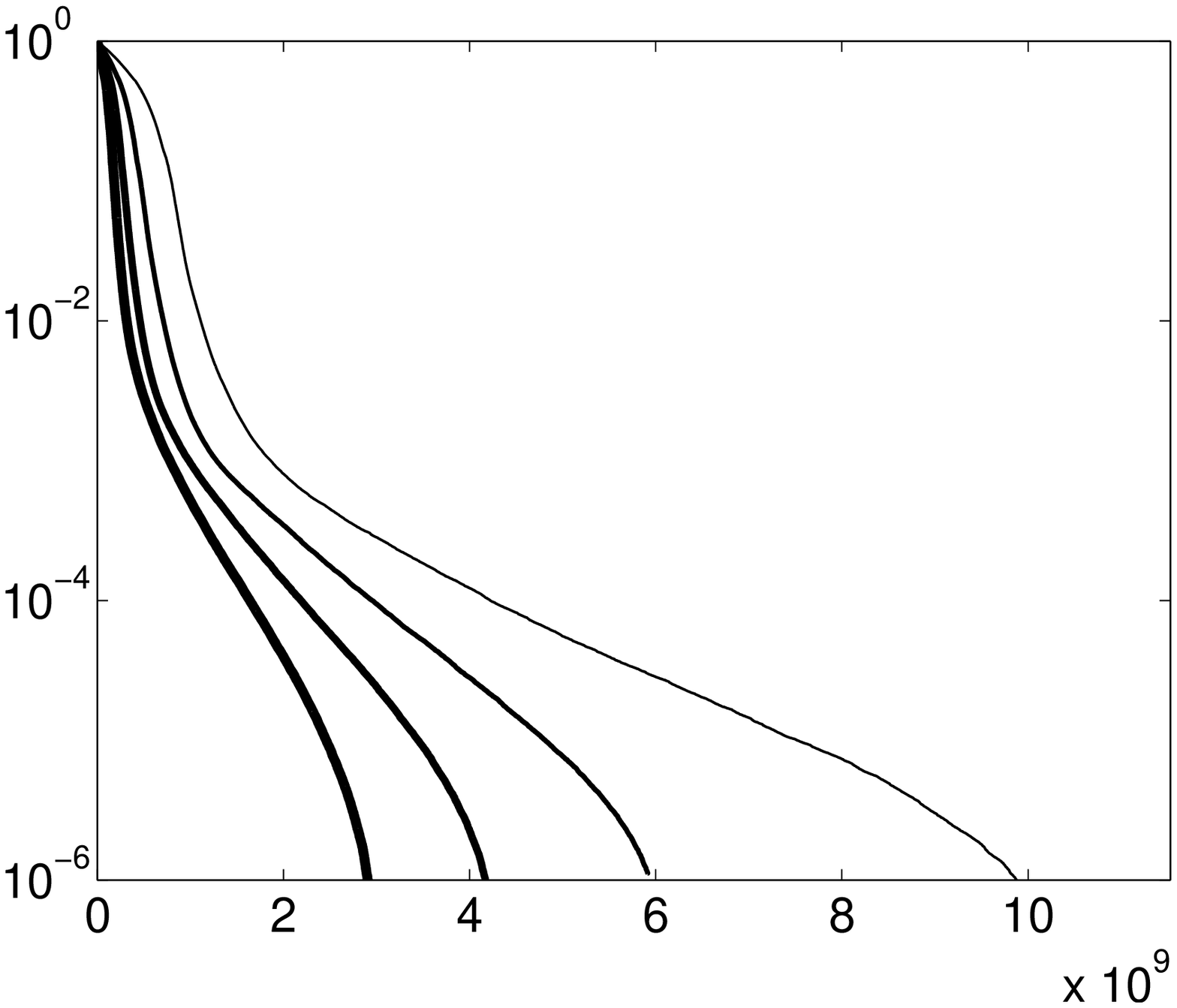}
  \end{minipage}
  \begin{minipage}[c]{.3\textwidth}
    \includegraphics[scale=0.2]{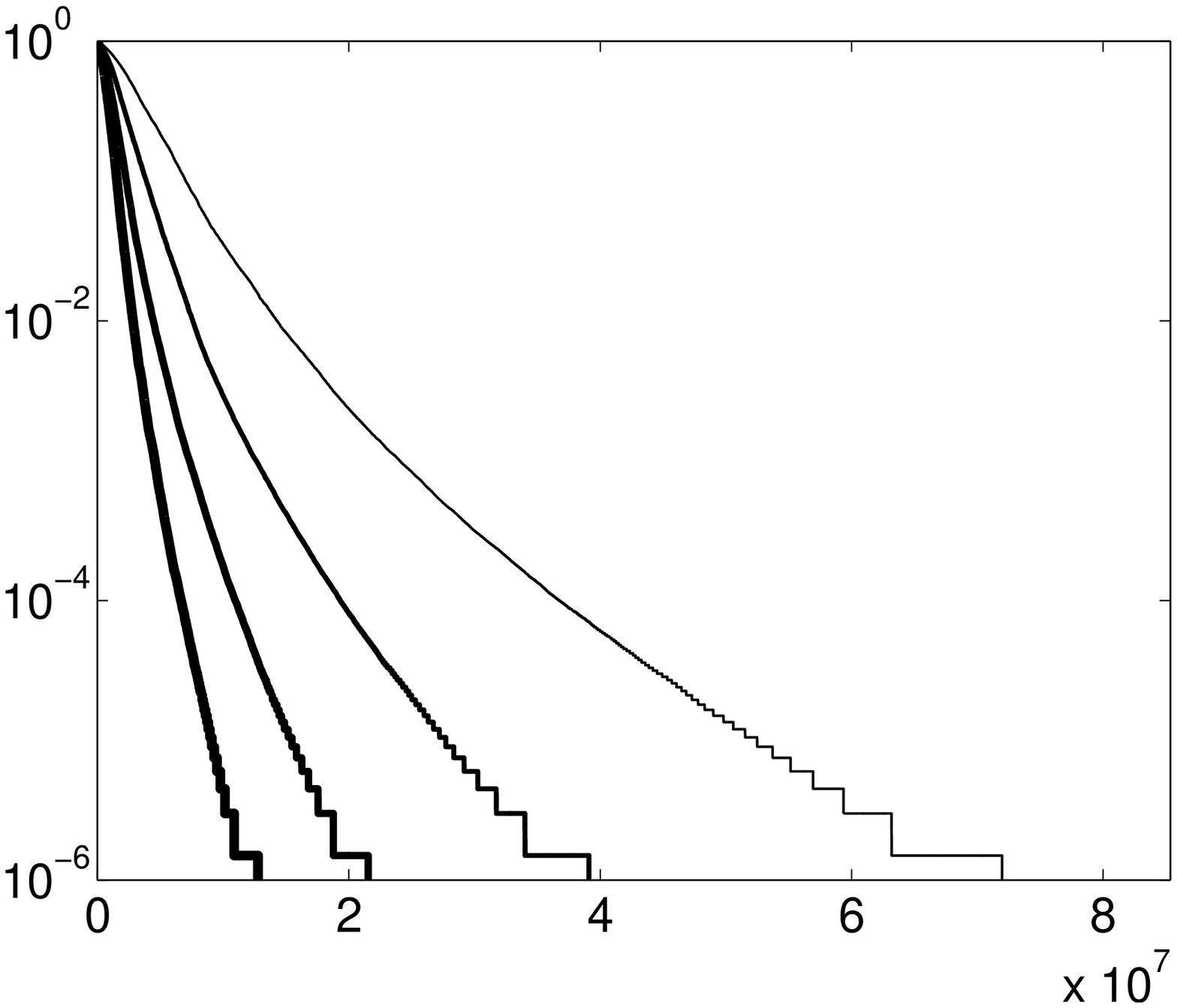}
  \end{minipage}
  \begin{minipage}[c]{.3\textwidth}
  \vspace{1.2cm}
  \begin{center}
   $\qquad$a9a
  \end{center}
  \end{minipage}
  \begin{minipage}[c]{.3\textwidth}
  \vspace{1.2cm}
  \begin{center}
   cod-rna
  \end{center}
  \end{minipage}
  \begin{minipage}[c]{.3\textwidth}
  \vspace{1.25cm}
   $\qquad$gisette scale
    \end{minipage}
 \vspace{-0cm}
 \hspace*{1.6em}
  \begin{minipage}[c]{.3\textwidth}
    \vspace{-2em}
    \includegraphics[scale=0.2]{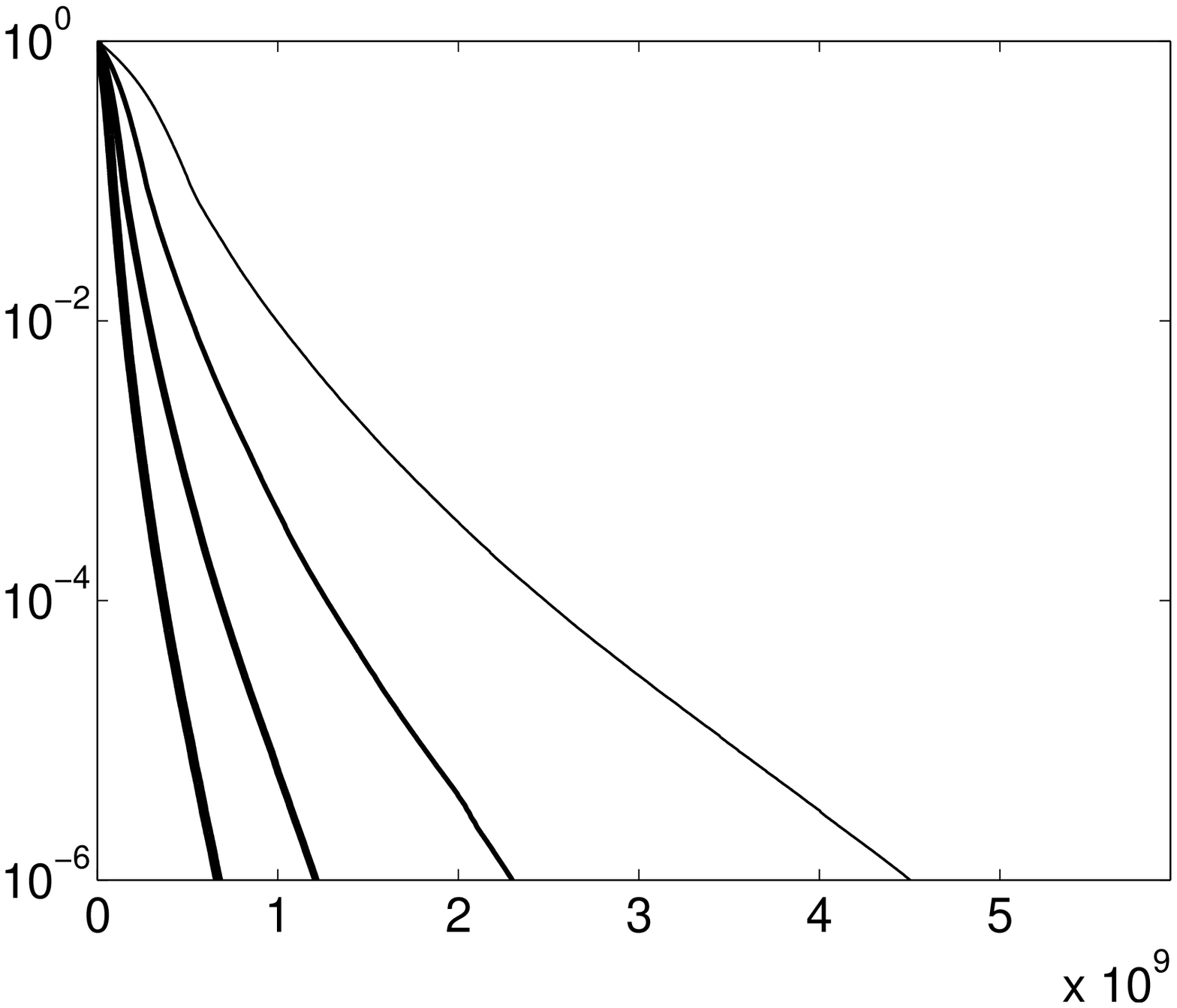}
  \end{minipage}
  \begin{minipage}[c]{.3\textwidth}
    \vspace{-2em}
    \includegraphics[scale=0.2]{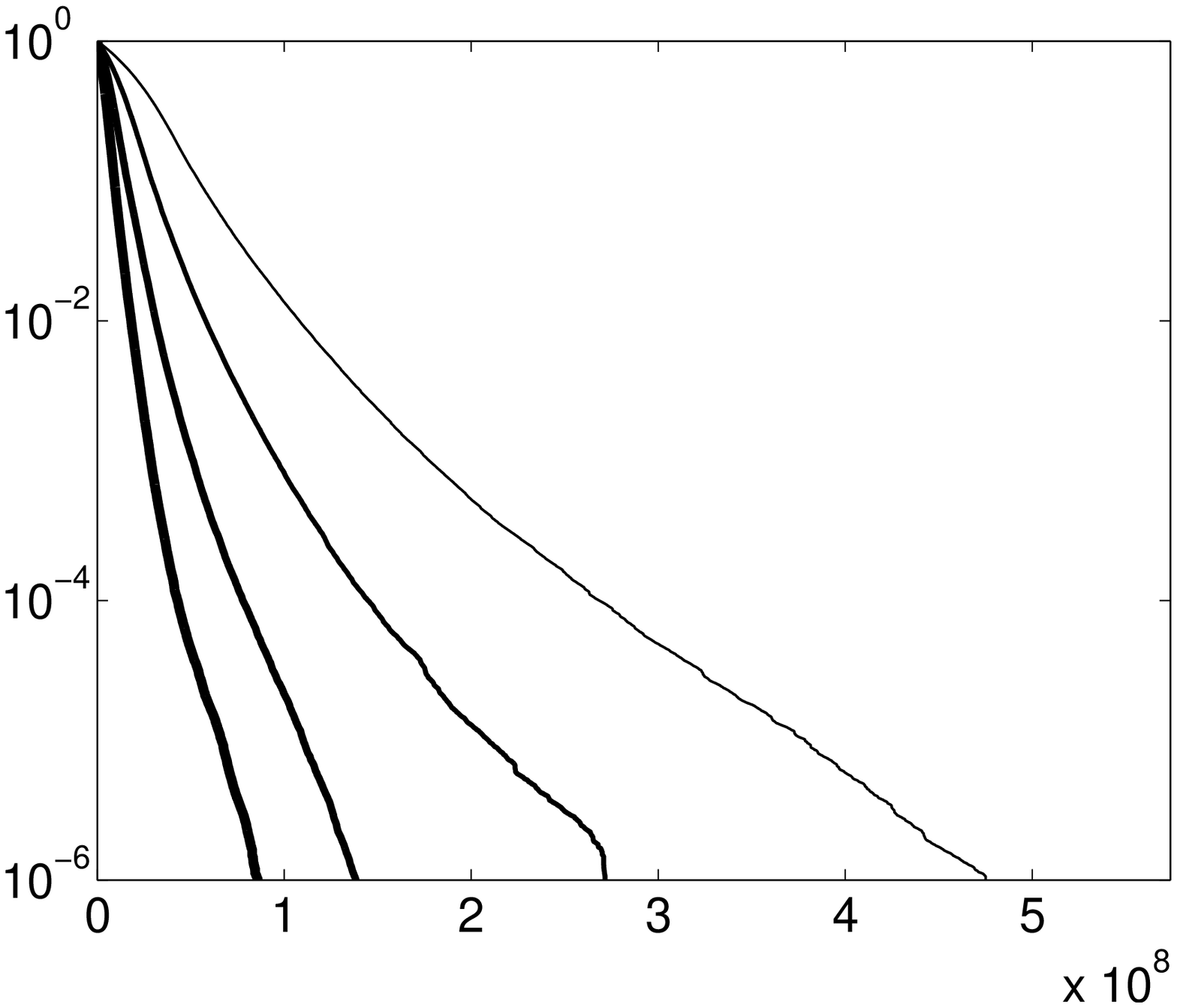}
  \end{minipage}
  \begin{minipage}[c]{.3\textwidth}
    \vspace{-2em}
    \includegraphics[scale=0.2]{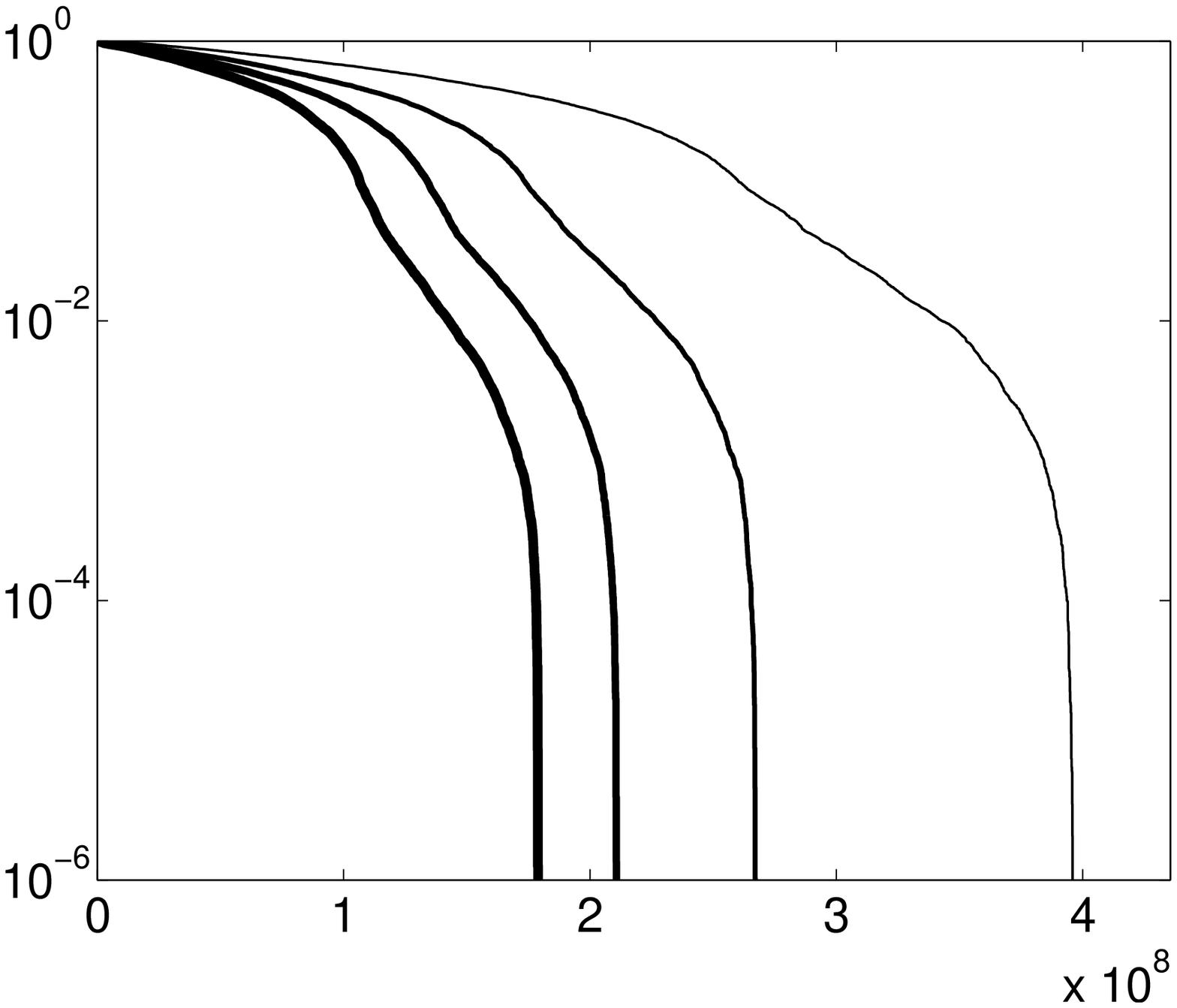}
  \end{minipage}
  \begin{minipage}[c]{.3\textwidth}
  \vspace{1.2cm}
  \begin{center}
   $\qquad$real-sim
  \end{center}
  \end{minipage}
  \begin{minipage}[c]{.3\textwidth}
  \vspace{1.2cm}
  \begin{center}
   rcv1
  \end{center}
  \end{minipage}
  \begin{minipage}[c]{.3\textwidth}
  \vspace{1.25cm}
   $\qquad\quad$w8a
    \end{minipage}
  \begin{minipage}[c]{1\textwidth}
    \vspace{2em}
	    \setlength{\unitlength}{0.4cm}
\begin{center}
 	    \begin{picture}(28.5,1)(0,0)
	    \put(0,0){\framebox(28.5,1)}
	    \put(2,0.4){\linethickness{0.2mm}\line(1,0){2}}
	    \put(4.5,0.3){$q=1$}
	    \put(9,0.4){\linethickness{0.35mm}\line(1,0){2}}
	     \put(11.5,0.3){$q=2$}
	    \put(16,0.4){\linethickness{0.50mm}\line(1,0){2}}
	    \put(18.5,0.3){$q=4$}
	    \put(23,0.4){\linethickness{0.65mm}\line(1,0){2}}
	    \put(25.5,0.3){$q=8$}
	    \end{picture}
	    \end{center}
  \end{minipage}
  \caption{PARSMO-1: Relative Error versus kernel evaluations per process. }\label{ReVSK}
\end{figure}
\clearpage
Our results show that the larger $q$ is, the steeper the $RE$ decrease is. This emphasizes the positive effect of moving along multiple SMO directions at a time.

As regards PARSMO-2, we note that, except for the MVP pair which can require the computation of the kernel columns  $\Qbold_{* i_1}$ and $\Qbold_{* j_1}$,
 each SMO process computes only the analytical solution of the two-dimensional subproblem, since kernel columns are already available in the cache.
Thus, PARSMO-2 may produce a cpu time saving even  by running the algorithm in a sequential fashion. In order to show the cheapness of its tasks, in Figure \ref{ReVST} we plot $RE$ versus the CPU-time consumed.
 \begin{figure}[ht!]
 \hspace{1.6em}
  \begin{minipage}[c]{.3\textwidth}
    \includegraphics[scale=0.2]{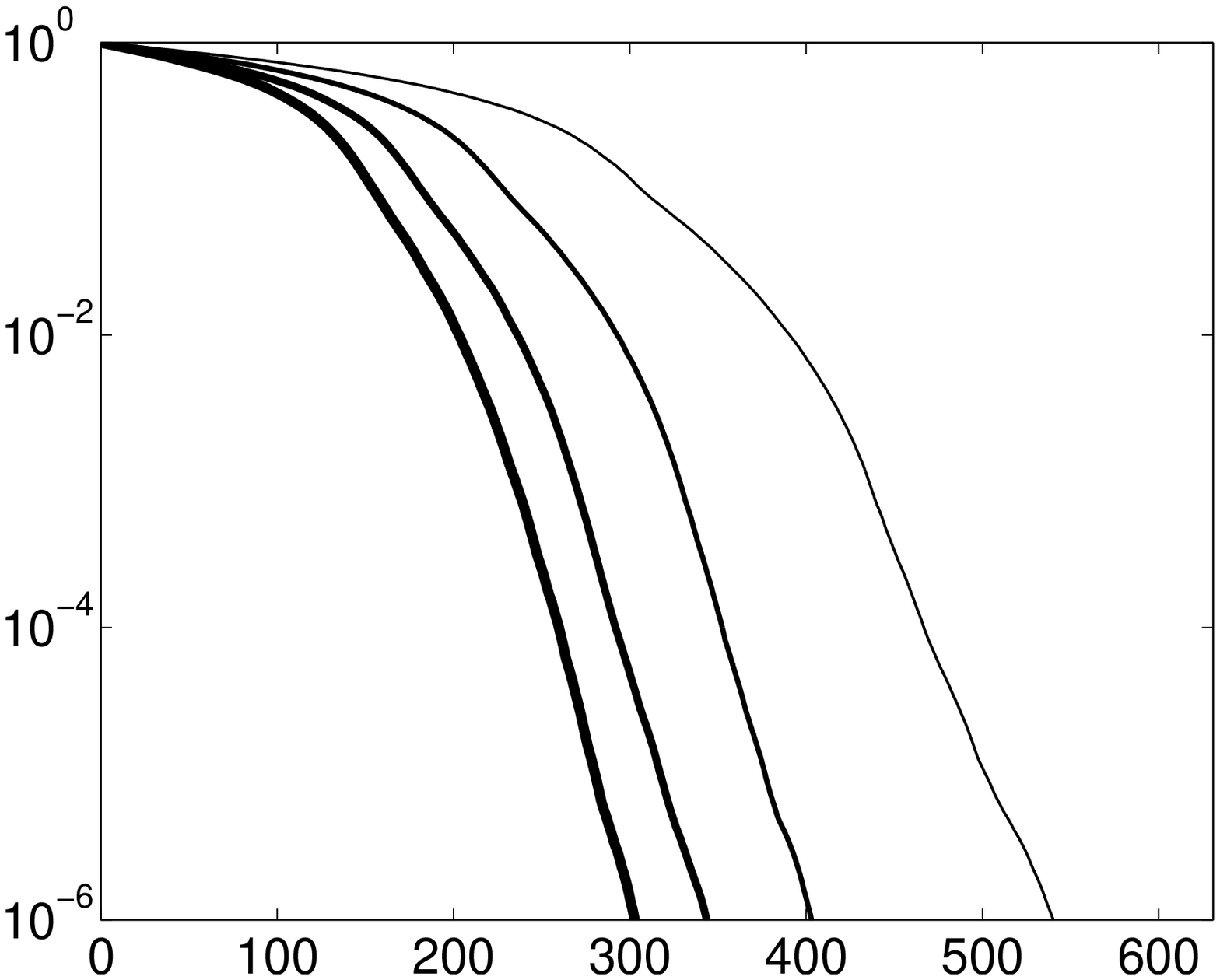}
  \end{minipage}
  \begin{minipage}[c]{.3\textwidth}
    \includegraphics[scale=0.2]{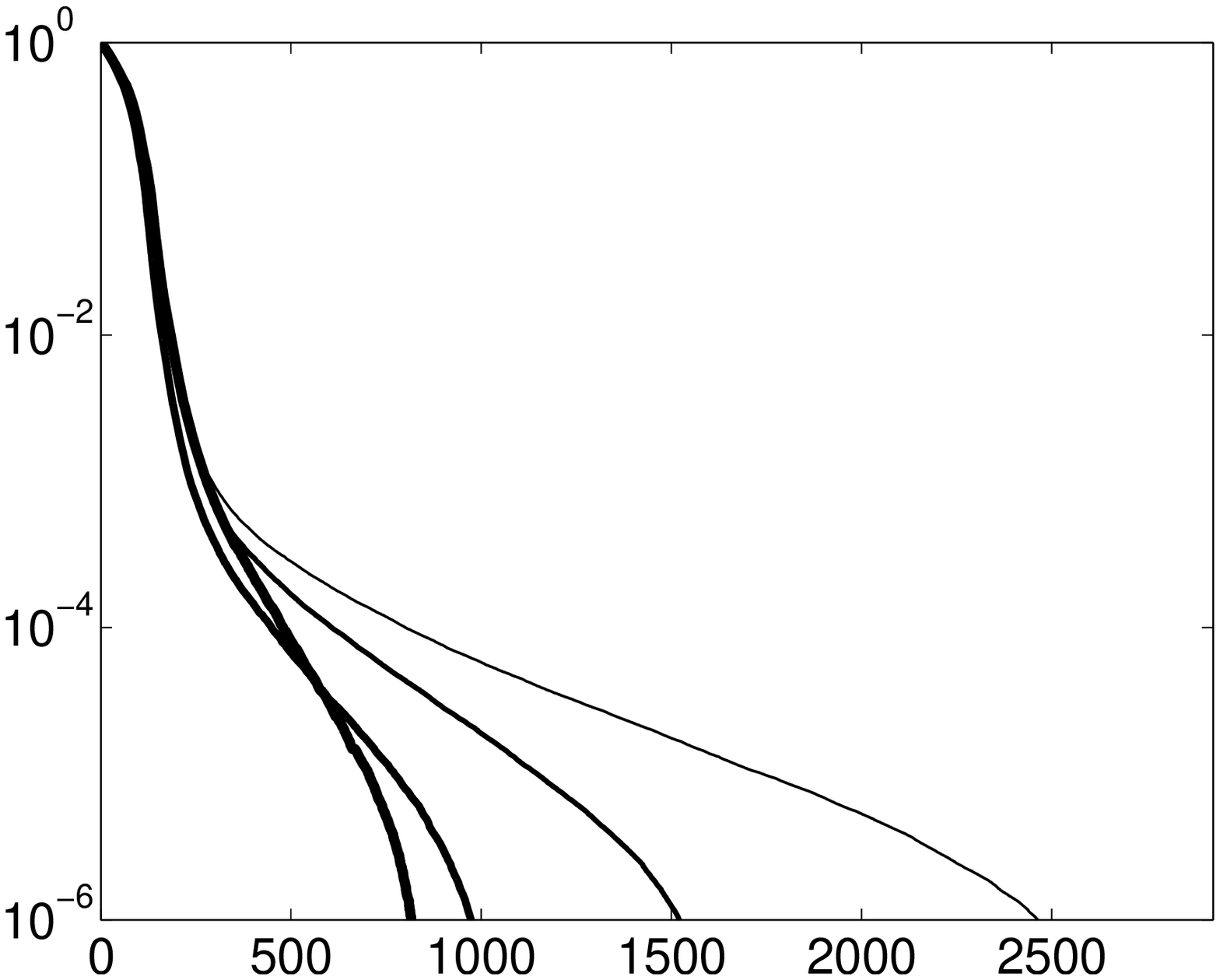}
  \end{minipage}
  \begin{minipage}[c]{.3\textwidth}
    \includegraphics[scale=0.2]{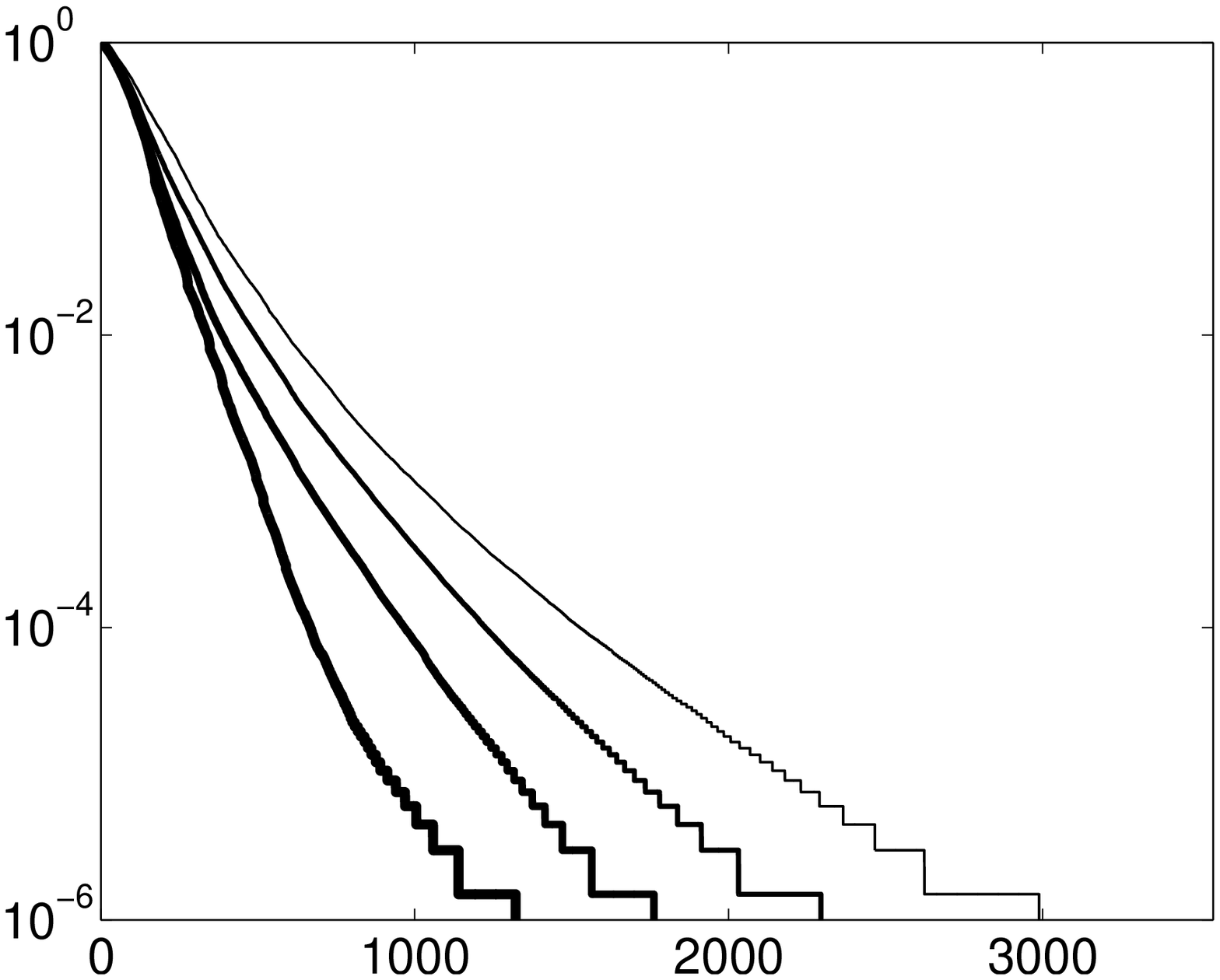}
  \end{minipage}
  \begin{minipage}[c]{.3\textwidth}
  \vspace{1.6cm}
  \begin{center}
   $\qquad$a9a
  \end{center}
  \end{minipage}
  \begin{minipage}[c]{.3\textwidth}
  \vspace{1.6cm}
  \begin{center}
   cod-rna
  \end{center}
  \end{minipage}
  \begin{minipage}[c]{.3\textwidth}
  \vspace{1.6cm}
   $\qquad$gisette scale
    \end{minipage}
 \vspace{-0cm}
 \hspace*{1.6em}
  \begin{minipage}[c]{.3\textwidth}
    \vspace{-2em}
    \includegraphics[scale=0.2]{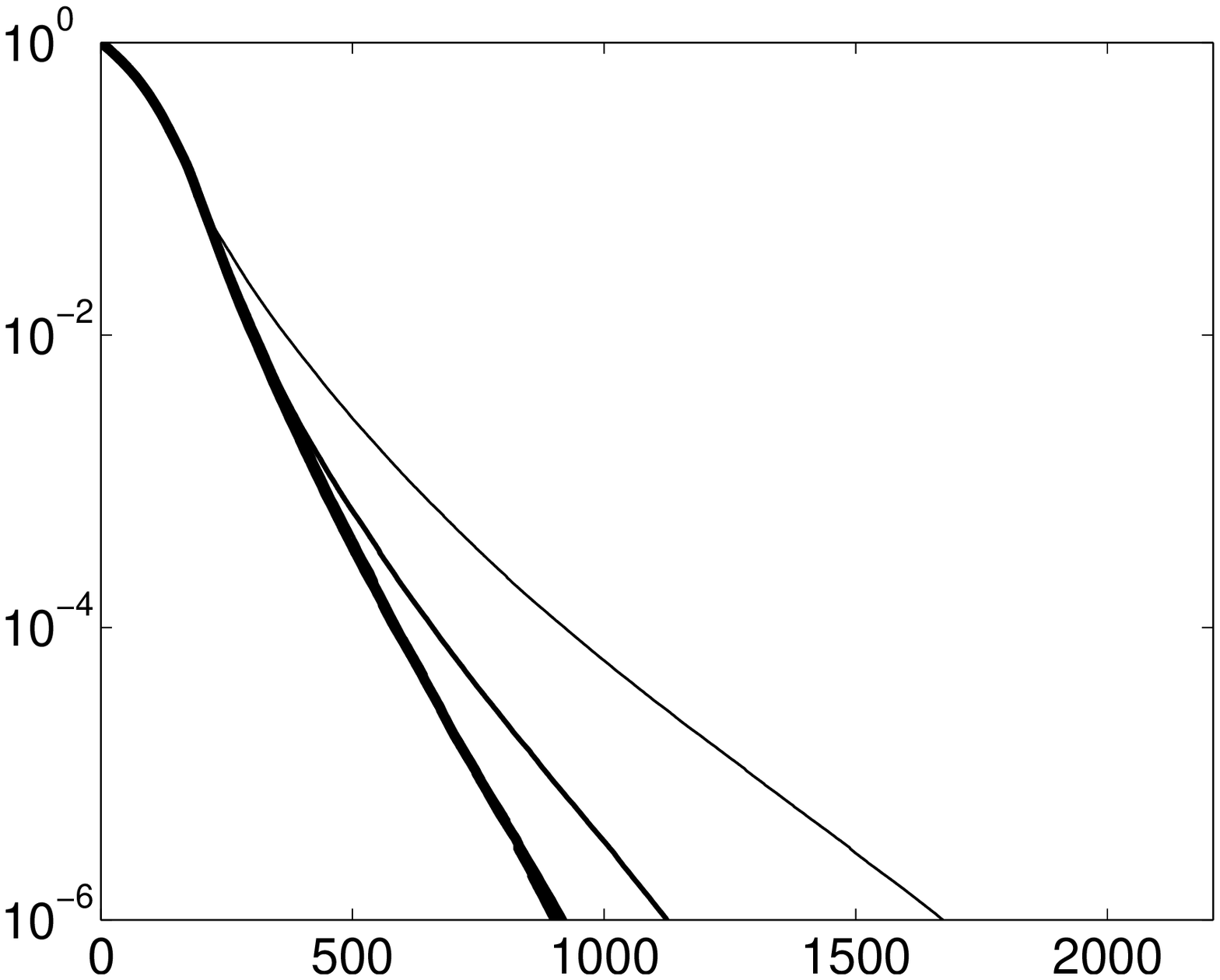}
  \end{minipage}
  \begin{minipage}[c]{.3\textwidth}
    \vspace{-2em}
    \includegraphics[scale=0.2]{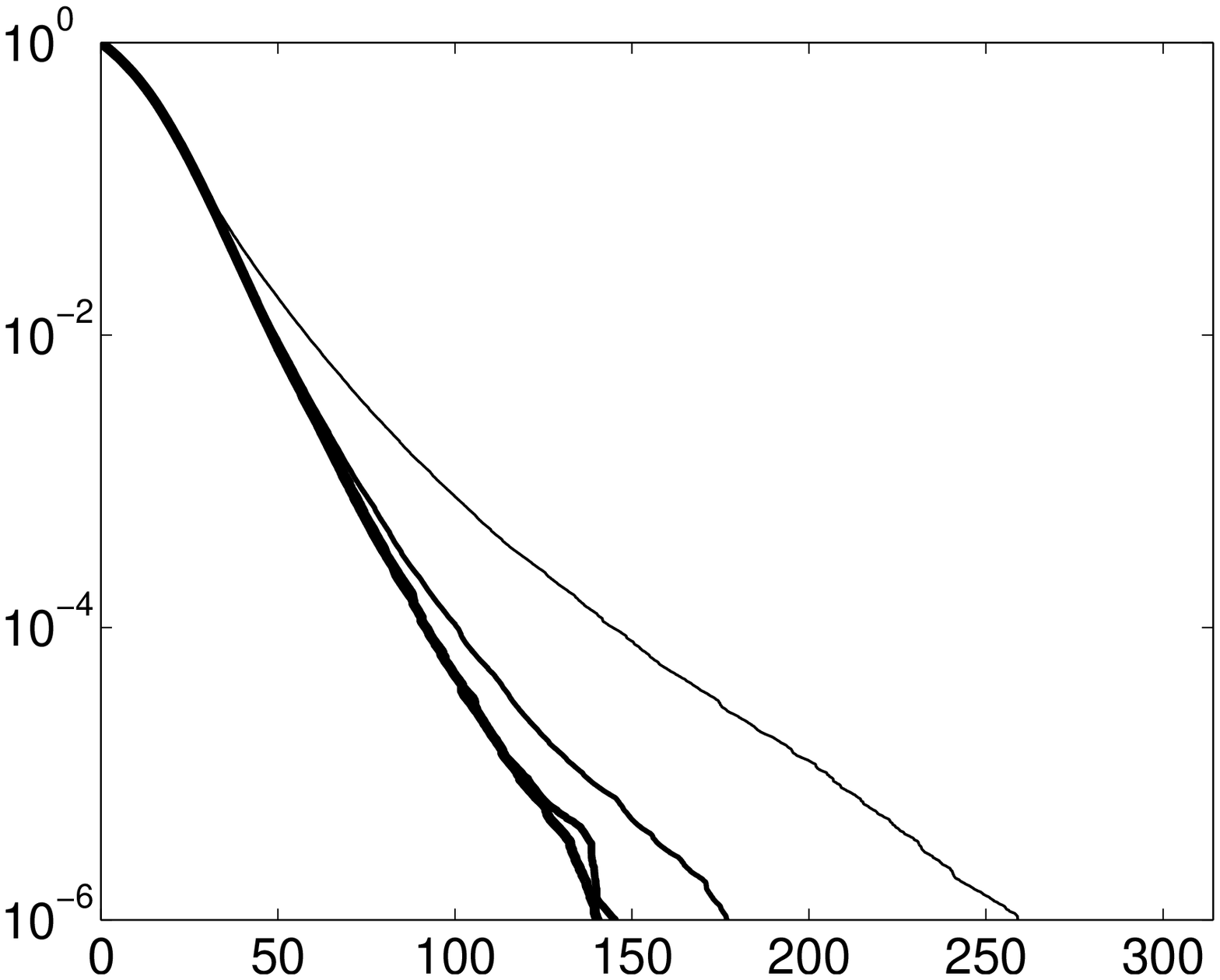}
  \end{minipage}
  \begin{minipage}[c]{.3\textwidth}
    \vspace{-2em}
    \includegraphics[scale=0.2]{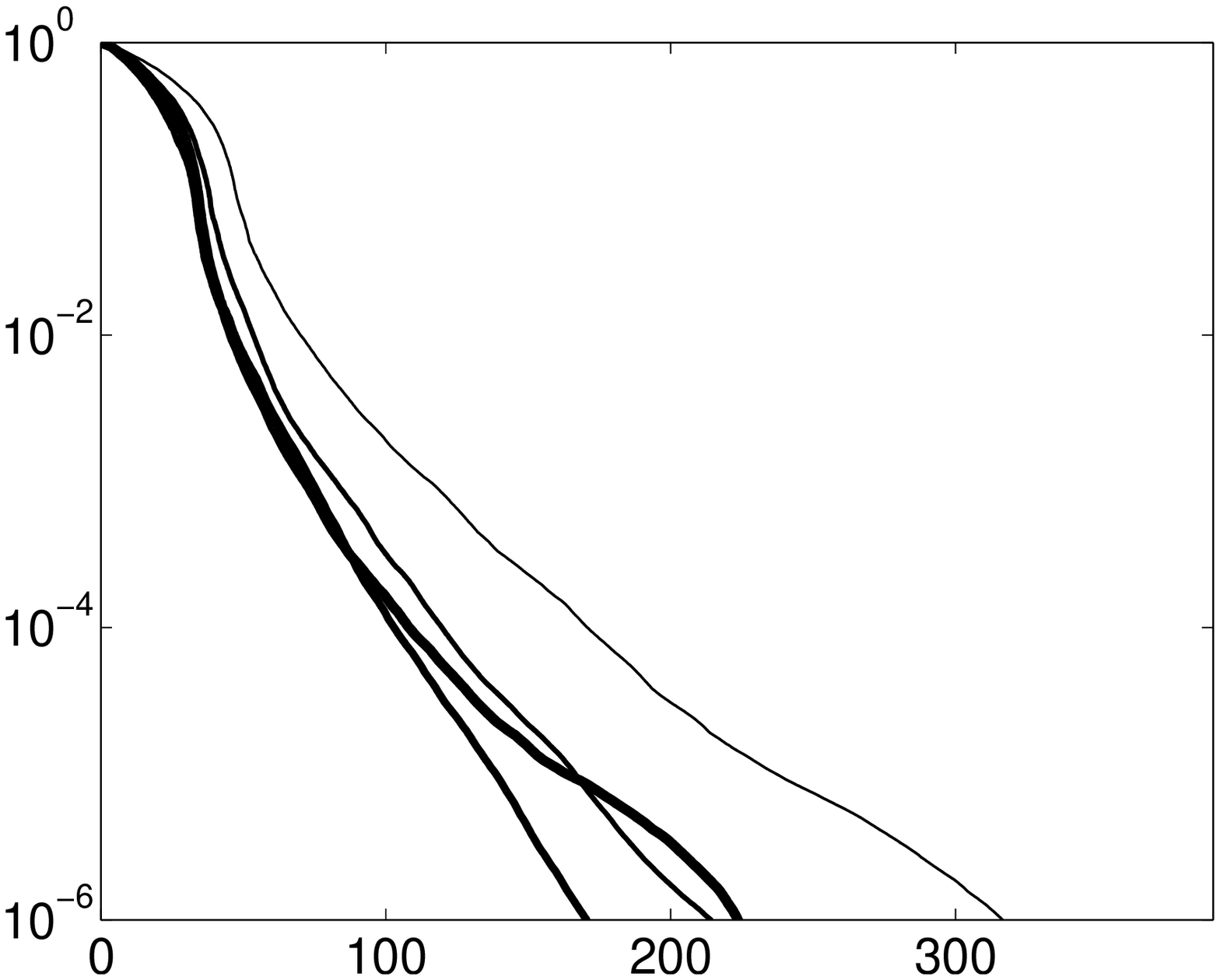}
  \end{minipage}
  \begin{minipage}[c]{.3\textwidth}
  \vspace{1.6cm}
  \begin{center}
   $\qquad$real-sim
  \end{center}
  \end{minipage}
  \begin{minipage}[c]{.3\textwidth}
  \vspace{1.6cm}
  \begin{center}
   rcv1
  \end{center}
  \end{minipage}
  \begin{minipage}[c]{.3\textwidth}
  \vspace{1.6cm}
   $\qquad\quad$w8a
    \end{minipage}
  \begin{minipage}[c]{1\textwidth}
    \vspace{2em}
	    \setlength{\unitlength}{0.4cm}
\begin{center}
 	    \begin{picture}(28.5,1)(0,0)
	    \put(0,0){\framebox(28.5,1)}
	    \put(2,0.4){\linethickness{0.2mm}\line(1,0){2}}
	    \put(4.5,0.3){$q=1$}
	    \put(9,0.4){\linethickness{0.35mm}\line(1,0){2}}
	     \put(11.5,0.3){$q=2$}
	    \put(16,0.4){\linethickness{0.50mm}\line(1,0){2}}
	    \put(18.5,0.3){$q=4$}
	    \put(23,0.4){\linethickness{0.65mm}\line(1,0){2}}
	    \put(25.5,0.3){$q=8$}
	    \end{picture}
	    \end{center}
  \end{minipage}
  \caption{PARSMO-2: Relative Error versus (sequential) CPU-time. \label{ReVST}}
\end{figure}

PARSMO-2 with $q>1$ seems to be faster than a classical MVP algorithm. This is due the use of multiple search directions without suffering from an increase of time consuming kernel evaluations or from the need of iterative solutions of larger quadratic subproblems.
It is important to outline that PARSMO-2 achieves its good performances by combining a convergent parallel structure with an efficient sequential implementation, and it seems to be useful also in a single-core environment.
\section*{Acknowledgment}
The authors thank Prof. Marco Sciandrone (Dipartimento di Ingegneria dell'Informazione, Università di Firenze) for fruitful discussions and suggestions that improved significantly the paper.

\bibliographystyle{plain}
\bibliography{parallel_svm}

\end{document}